\documentclass[11pt,reqno]{amsart}
\usepackage[toc,page]{appendix}
\usepackage{amsmath,amsfonts,amssymb,amsthm,mathtools}
\usepackage{thmtools}
\usepackage{thm-restate}
\usepackage[vcentermath]{youngtab}
\usepackage[all]{xy}
\usepackage{csquotes}
\usepackage[usenames,dvipsnames]{xcolor}
\usepackage[demo]{graphicx}
\usepackage{subcaption}
\usepackage{enumerate}
\usepackage{pstricks}
\usepackage{filecontents}
\usepackage{graphicx}
\usepackage{cite}
\usepackage{setspace}
\usepackage{tabu}
\usepackage{tikz}
\usetikzlibrary{calc}
\usetikzlibrary{decorations.markings}
\usetikzlibrary{patterns}
\usepackage[margin=1in]{geometry}
\usepackage{siunitx} %
\usepackage{hyperref}
\usepackage{cleveref}
\hypersetup{colorlinks,  citecolor=blue,  linkcolor=BrickRed, urlcolor=black}

\usetikzlibrary{chains}
\usepackage{indentfirst}
\usepackage{comment}
\usepackage{todonotes}
\DeclareFontFamily{OT1}{rsfs}{}
\DeclareFontShape{OT1}{rsfs}{n}{it}{<-> rsfs10}{}
\DeclareMathAlphabet{\mathscr}{OT1}{rsfs}{n}{it}

\DeclareMathOperator{\Reg}{Reg}

\DeclareMathOperator{\Tr}{T}

\DeclareMathOperator{\Core}{Core}

\DeclareMathOperator{\res}{res}
\DeclareMathOperator{\reg}{Reg}
\DeclareMathOperator{\sgn}{sgn}
\DeclareMathOperator{\M}{M}
\DeclareMathOperator{\Cr}{Colreg}

\DeclareMathOperator{\X}{X}
\DeclareMathOperator{\I}{I}
\DeclareMathOperator{\J}{J}

\DeclareMathOperator{\Sr}{Colseg}

\makeatletter
\renewcommand{\boxed}[1]{\text{\fboxsep=.2em\fbox{\m@th$\displaystyle#1$}}}
\makeatother

\newcommand{\con}{\textrm{-}}

\newcommand\hsl{\widehat{\mathfrak{sl}}}
\newcommand{\properideal}{%
  \mathrel{\ooalign{$\lneq$\cr\raise.22ex\hbox{$\lhd$}\cr}}}

\theoremstyle{plain}
\newtheorem{Theorem}{Theorem}[section]
\newtheorem{Proposition}[Theorem]{Proposition}
\newtheorem{Lemma}[Theorem]{Lemma}

\newtheorem{Conjecture}[Theorem]{Conjecture}
\newtheorem{Example}[Theorem]{Example}

\theoremstyle{definition}
\newtheorem{Definition}[Theorem]{Definition}

\theoremstyle{remark}
\newtheorem{Remark}[Theorem]{Remark}

\usepackage[foot]{amsaddr}

\makeatletter
\renewcommand{\email}[2][]{%
  \ifx\emails\@empty\relax\else{\g@addto@macro\emails{,\space}}\fi%
  \@ifnotempty{#1}{\g@addto@macro\emails{\textrm{(#1)}\space}}%
  \g@addto@macro\emails{#2}%
}
\makeatother

\title{The title}%

\author{Allen Wang}
\address[A.Wang]{Department of Mathematics\\ Massachusetts Institute of Technology\\ Cambridge, MA, 02139, USA}
\email[A. Wang]{awang23@mit.edu}

\author{Guangyi Yue}
\address[G. Yue]{Department of Mathematics\\ Massachusetts Institute of Technology\\ Cambridge, MA, 02139, USA}
\email[G. Yue]{gyyue@mit.edu}

\begin{document}

\setcounter{tocdepth}{4}

\title[Mullineux Involution and the Generalized Regularization]{Relationship Between Mullineux Involution and the Generalized Regularization}
\maketitle
\begin{abstract}
The Mullineux involution is an important map on $p$-regular partitions that originates from the modular representation theory of $\mathcal{S}_n$. In this paper we study the Mullineux transpose map and the generalized column regularization and prove a condition under which the two maps are exactly the same. Our results generalize the work of Bessenrodt, Olsson and Xu, and the combinatorial constructions is related to the Iwahori-Hecke algebra and the global crystal basis of the basic $U_q(\hsl_b)$-module. In the conclusion, we provide several conjectures regarding the $q$-decomposition numbers and generalizations of results due to Fayers.
\end{abstract}

\tableofcontents

\section{Introduction}
The Mullineux involution appears in the study of modular representations of symmetric groups, and the difficult combinatorial construction of the Mullineux involution is known to admit a conceptual description in terms of modular representations and crystals. Since the irreducible $p\con$modular representations of the symmetric group $\mathcal{S}_n$ are labeled by the $p\con$regular partitions of $n$, our study of the Mullineux involution as an  algebraic operator is related to the underlying combinatorics of partitions. Given an irreducible representation $\rho_\lambda$, the new representation obtained by taking the tensor product with the one-dimensional sign representation $\sgn$ corresponds to the Mullineux involution of $\lambda^{\M_p}$. This notion defines $\lambda^{\M_p}$ via the equation:
$$\rho_{\lambda^{\M_p}}=\rho_\lambda\otimes\sgn.$$ 

There are a few combinatorial definitions of $\M_p$ in \cite{kleshchev1996branching,ford1997proof} where $p$ is not necessarily prime, and all of them are highly non-trivial. In this paper, we focus on the combinatorics of $\M_p$, so we use the parameter $b$ instead of $p$ since our parameters are generally positive integers.

Walker, Bessenrodt, Olsson, Xu, and Fayers studied the combinatorial properties of Mullineux involution by relating it to another operation named (resp. column) regularization $\Reg_{b}$ (resp. $\Cr_{b}$). These two operations are naturally related by the Iwahori-Hecke algebra $\mathcal{H}=\mathcal{H}_{\mathbb{F},q}(S_n)$ of the symmetric group where $\mathbb{F}$ is a field and $q\in\mathbb{F}\setminus\{0\}$. Let $b$ be the minimal integer satisfying $1+q+\cdots+q^{b-1}=0$ and the decomposition numbers $d_{\lambda,\mu}=\left[S^{\lambda}:D^{\mu}\right]$ be the multiplicity of the simple module $ D^{\mu}$ ($\mu$ is $b$-regular) inside the Specht module $S^{\lambda}$ of $\mathcal{H}$. Previous works show that the identities for these decomposition numbers involve certain combinatorics of partitions. James proved that $d_{\lambda,\mu} = 0 $ unless $\mu\unlhd\lambda^{\Reg_{1,b}}$ and $d_{\lambda,\lambda^{\Reg_{b}}}=1$ \cite{james1976decomposition}. And later, Lascoux, Leclerc and Thibon  proved the identity $d_{\lambda,\mu}=d_{\lambda^{\Tr},\mu^{\M_b}}$ \cite{Lascoux1996}. Therefore, it is natural to study the relationship between the Mullineux involution and the (resp. column) regularization map. In combinatorics, Walker proved that when the partition $\lambda$ is horizontal or row-stable, it satisfies $\lambda^{\M_b\Tr} = \lambda^{\Cr_{b}}$ \cite{walker1994modular,walker1996horizontal}. Later, Bessenrodt, Olsson, and Xu showed in \cite{bessenrodt1999properties} that Walker's conditions can be broadened to the short-legged (or shallow) partitions, namely for every hook in the partition divisible by $b$, the length of the corresponding arm is at least $(b-1)$ times that of the leg. Even better, Bessenrodt, Olsson, and Xu proved that those partitions are the only ones satisfying $\lambda^{\M_b\Tr} = \lambda^{\Cr_{b}}$. Fayers generalized Bessenrodt, Olsson, and Xu's result by considering partitions that are not necessarily $b$-regular. The Specht module $S^\lambda$ is reducible when $\lambda^{\Reg_b\M_b} \neq\lambda^{\Tr\Reg_{b}}$ and Fayers proved in \cite{fayers2008regularisation} that the identity $\lambda^{\Reg_b\M_b} = \lambda^{\Tr\Reg_{b}}$ holds if and only if hooks divisible by $b$ must be either shallow or steep. However, in all these previous works, the operators $\Reg_b$ and $\Cr_b$ only involved a single parameter $b$. 

A second parameter was added by Dimakis and the second author in \cite{Dimakis2018} where the parameters of (resp. column) regularization $\Reg_{a,b}$ (resp. $\Cr_{a,b}$) were extended to any positive rational number $\frac{a}{b}$ in the unit interval. In \cite{Dimakis2018}, the composition of a certain series of column regularizations and Mullineux transposes were shown to be same when applied to the one-row partition, giving a series of monotonically decreasing partitions. On one hand, this result gives a special situation where the simpler operation, column regularization, can be used to understand the Mullineux map, whose combinatorial definition is more convoluted; on the other hand, it proves a special case of Bezrukavnikov's conjecture stated in the appendix of \cite{Dimakis2018}. In this paper, we extend the idea of choosing a suitable parameter $a$ and finding the condition under which Mullineux transpose is identical to the generalized column regularization, generalizing the result of Bessenrodt, Olsson, and Xu in \cite{bessenrodt1999properties}, meanwhile shedding light on other cases of Bezrukavnikov's conjecture. The main theorem is as follows:

\begin{restatable}{thm}{mainthm}\label{theorem generalizing BOX}
Given positive integers $a<b$, let $\lambda$ be a partition satisfying $\lambda^{\Cr_{a,b}}\in \mathcal{P}$ such that all hooks $H_{i,j}$ in $\lambda$ with $b\mid H_{i,j}$ satisfy:
\begin{equation}\label{equation: shallow}
    \left(\frac{b}{a}-1\right)l_{i,j} < a_{i,j} +1,
\end{equation}
then $\lambda$ is $b\con$regular and $\lambda^{\M_b\Tr} = \lambda^{\Cr_{a,b}}$.
\end{restatable}

The theorem is proved combinatorially in Section 3 by analyzing the Young diagrams. The proof follows after completely characterizing the shape of the partitions satisfying the inequalities in Equation \eqref{equation: shallow}.

The rest of the paper is organized as follows. Section 2 is an overview of combinatorics involved. Section 3 presents a detailed description of partitions satisfying the conditions in Theorem \ref{theorem generalizing BOX} and concludes with the proof Theorem \ref{theorem generalizing BOX}. Section 4 first recalls the basic facts about Iwahori-Hecke algebras that motivates the representation-theoretic significance of the main result. The section concludes with conjectures regarding the converse of Theorem \ref{theorem generalizing BOX} and a generalization of Fayers's main theorem from \cite{fayers2008regularisation}.

{\bf Acknowledgements.} The authors thank Roman Bezrukavnikov and Richard Stanley for useful conversations and comments. They would also like to thank Panagiotis Dimakis for discussions and proofreadings of the manuscript. The first author is grateful to the MIT-PRIMES program for facilitating a part of this research.

\section{Preliminaries}

In this section, we introduce the fundamental vocabulary and notation needed for the remainder of the paper.

A \textit{partition} $\lambda=(\lambda_1,\ldots, \lambda_k)$ of $n\in \mathbb{N}$ is a tuple of non-increasing positive integers, i.e. $\lambda_1\ge\cdots\ge \lambda_k>0$, and $|\lambda| = \sum_{i=1}^k\lambda_i = n$ is called the \textit{size} of $\lambda$. When appropriate, we also append infinite zeros at the end of $\lambda$, i.e., $\lambda_{k+1}=\lambda_{k+2}=\cdots=0$. We denote $\mathcal{P}_n$ to be the set of all partitions of $n$ and $\mathcal{P}=\bigcup_{n\geq0}\mathcal{P}_n$. Denote $l(\lambda) = k$ to be the number of nonzero parts of $\lambda$. Given two positive integers $i\leq j$, denote $\lambda_{[i,j]}$ to be the subpartition $(\lambda_i,\ldots,\lambda_j)$. Given a positive integer $b$, we say $\lambda$ is \textit{$b$-regular} if there is no index $ i$ such that $\lambda_i =\cdots= \lambda_{i+b-1}>0$. And for simplicity, we define the \textit{concatenation} of two finite positive integer sequences $\lambda$ and $\mu$ as the tuple $\lambda\oplus\mu = (\lambda_1,\ldots,\lambda_{l(\lambda)},\mu_1,\ldots,\mu_{l(\mu)})$.

We identify each partition with its corresponding \textit{Young diagram}. In this paper, except Section 2.3, we adopt the \textit{English convention} for the Young diagram: rotate the plane to orient the \textit{positive $x$-axis pointing south} and the \textit{positive $y$-axis pointing east}. Recall that the Young diagram consists of \textit{boxes}, which are unit squares parallel to the axis with vertices in integer coordinates. For the sake of notation, we identify a box and its southeast vertex by the same name $(i,j)$ where a box $(i,j)\in \lambda$ iff $j\le \lambda_i$.

The transpose (or conjugate) $\lambda^{\Tr}$ of a Young diagram $\lambda$ is given by: $$\{(i,j)\in\mathbb{N}\times\mathbb{N}: 1\le j, 1\le i \le \lambda_j\}.$$ 

A box $A=(i,j)\in\lambda$ is called a \textit{removable} box of $\lambda$ if $\lambda\setminus A\in\mathcal{P}_{n-1}$. A box $B\notin\lambda$ is called an \textit{addable} box of $\lambda$, if $\lambda\cup B\in\mathcal{P}_{n+1}$. Fix a positive integer $b$, the \textit{residue} of $A$ with respect to $b$, denoted by $\res_b A$, is defined to be the residue class $(j-i)$ mod $b$. %We will say a box is an \textit{$i$-box}, if its residue is $i$. 

Given a box $(i,j)\in \lambda$, the corresponding \textit{arm} $a_{i,j}=a_{i,j}(\lambda)$ is the set of boxes $(i,j')\in\lambda$ with $j<j'$. We use $a_{i,j}$ to denote either this set of boxes or the number of elements of the above set interchangeably. Similarly, the \textit{leg} $l_{i,j}=l_{i,j}(\lambda)$ is the set of boxes $(i',j)\in\lambda$ with $i<i'$. We use $l_{i,j}$ to denote either the above set or the number of elements of the above set interchangeably as well. Finally the \textit{hook} $H_{i,j}=H_{i,j}(\lambda)$ is the union of sets $\{(i,j)\}\cup a_{i,j}\cup l_{i,j}$. Again, the number of elements of the hook is also denoted by $H_{i,j}$ and is equal to $1+a_{i,j}+l_{i,j}$. The northeast-most (resp. southwest-most) box in $H_{i,j}$ is called the hand (resp. foot) box associated to $(i,j)$, denoted by $\mathfrak{h}_{i,j}=\mathfrak{h}_{i,j}(\lambda)=(i,j+a_{i,j}$) (resp. $\mathfrak{f}_{i,j}=\mathfrak{f}_{i,j}(\lambda)=(i+l_{i,j},j)$).

There are two special classes of hooks of particular interest:

\begin{Definition}\label{definition of shallow and steep hooks}
Given two positive integers $a<b$ and a partition $\lambda$. A hook $H_{i,j}$ in $\lambda$ is \textit{$(a,b)$-shallow} if it satisfies:

\begin{equation}\label{eq: shallow hook}
    \left(\frac{b}{a}-1\right)l_{i,j} < a_{i,j} +1.
\end{equation}
\\
Dually, a hook $H_{i,j}$ is \textit{$(a,b)$-steep} if it satisfies:

\begin{equation}\label{eq: steep hook}
    \left(\frac{b}{a}-1\right)a_{i,j} < l_{i,j} +1.
\end{equation}

\end{Definition}

\begin{Remark}\label{shallow-steep}
When $b\mid H_{i,j}$, $H_{i,j}$ is $(a,b)$-shallow iff $l_{i,j}\le ta-1$ and $a_{i,j}\ge t(b-a)$ for some $t\in \mathbb{N}_{>0}$; and similarly, $H_{i,j}$ is $(a,b)$-steep iff $a_{i,j}\le ta-1$ and $l_{i,j}\ge t(b-a)$ for some $t\in \mathbb{N}_{>0}$. In particular, the above inequalities hold when $t = H_{i,j}/b.$
\end{Remark}

\begin{Lemma}
\label{simultaneously shallow/steep}
A hook divisible by $b$ can be both $(a,b)$-shallow and $(a,b)$-steep only if $b<2a$.
\end{Lemma}
\begin{proof}
By the above remark, by taking $t = H_{i,j}/b$, we know $t(b-a)\leq a_{i,j}\leq ta-1$ and $t(b-a)\leq l_{i,j}\leq ta-1$ for some $t\in\mathbb{N}_{>0}$. Hence $t(b-a)\le ta-1$, and we conclude $b<2a$.
\end{proof}
%% dominance not needed? for the first part of the proof
Finally, for two partitions $\lambda$ and $\mu$ of the same size, the dominance order is defined as $\lambda\unlhd\mu$ if $\sum_{i=1}^k\lambda_i\leq\sum_{i=1}^k\mu_i$ is satisfied for all $k$. Note that $\lambda\unlhd\mu$ iff $\mu^{\Tr}\unlhd\lambda^{\Tr}$.

\subsection{Mullineux Transpose}
We abbreviate the composition of Mullineux involution and transpose as \textit{Mullineux transpose}. There are multiple definitions for Mullineux transpose, but here we will follow the approach of Bessenrodt, Olsson and Xu from \cite{bessenrodt1999properties}.

%First we define $\lambda^{\R}$ (resp. $\lambda^{\C}$) be the resulting partition from removing the largest row (resp. column) from $\lambda.$ More explicitly, for $\lambda=(\lambda_1,\ldots, \lambda_k)$, $\lambda^{\R}=(\lambda_2,\ldots, \lambda_k)$ and $\lambda^{\C}=(\lambda_1-1,\ldots, \lambda_k-1)$ (with zero parts removed if necessary). We also have $\lambda^{\R} = \lambda^{\Tr \C \Tr}.$

\begin{Definition}\label{definition of rim and b-rim}
The \textit{rim} of a partition $\lambda$ is the set of boxes $\{(i,j)\in\lambda \mid (i+1,j+1)\notin \lambda\}$. If  $\lambda$ is $b$-regular, we define its \textit{$b$-rim} to be the a subset of its rim obtained through the following procedure:

The $b\con$rim consists of the several pieces where each piece, \textit{except possibly the last one}, contains $b$ boxes. We choose the first $b$ boxes from the rim, beginning with the east-most box of the first row and moving contiguously southwestwards in the rim of $\lambda$. If the last box of this piece is chosen from the $i_0\con$th row of $\lambda$, then we choose the second piece of $b$ boxes beginning with the right-most box of the next row ${i_0+1}$. Continue this procedure until we reach the last piece ending in the last row.

We call a maximal set of contiguous boxes of the $b$-rim a \textit{segment}. When the number of boxes in this segment is a multiple $b$, it is a \textit{$b$-segment}. Otherwise, it is called a \textit{$b'$-segment}. By construction, every $\lambda$ has at most one $b'$-segment.

Finally, denote $\lambda^{\I_b}$ to be the partition obtained by removing the $b$-rim from $\lambda$.

%\begin{Remark}
%Note that there is a correspondence between consecutive boxes on the rim and hooks in $\lambda$. Given a hook, the corresponding boxes on the rim are those that lie between the hand and foot boxes. Given a set of consecutive boxes on the rim, the corresponding hook is determined by setting the hand node as the northeast-most box in the set and the foot node as the southwest-most box in the set.
%\end{Remark}

\end{Definition}

\begin{Definition}\label{definition of Mullineux + transpose}

Given $\lambda=(\lambda_1,...,\lambda_k)$ and $\lambda^{\I_b}=(\mu_1,...,\mu_k)$, where $k=l(\lambda)$ and some of the $\mu_i$'s at the end of $\lambda^{\I_b}$ are allowed to be zero, define $$\lambda^{\J_b}\coloneqq(\mu_1+1,...,\mu_{k-1}+1,\mu_k+\delta),$$ where $\phi(\lambda)=|\lambda|-|\lambda^{\I_b}|$ and
\[\delta=
\begin{cases}
0\text{ if } b\nmid \phi(\lambda)\\
1\text{ if } b\mid \phi(\lambda)
\end{cases}. \]

The set of boxes $\lambda\setminus\lambda^{\J_b}$ is called the \textit{truncated $b$-rim} of $\lambda$.
\end{Definition}

\begin{Definition}
\label{BOXmullineux}
For a $b\con$regular partition $\lambda$, the operator $\X_b$ is defined as $\lambda^{\X_b}\coloneqq(j_1,j_2,\ldots)$, where 
$$j_i=|\lambda^{\J_b^{i-1}}|-|\lambda^{\J_b^i}|.$$
The \textit{Mullineux map} $\M_b$ for a $b$-regular partition is defined to be $\lambda^{\M_b}=\lambda^{\X_b\Tr}$.

Recursively, we can write 
\begin{equation}\label{eq: M recursion}
    \lambda^{\M_b\Tr} = (|\lambda|-|\lambda^{\J_b}|)\oplus \lambda^{\J_b\M_b\Tr}.
\end{equation}

\end{Definition}

\begin{Remark}
In case $b$ is a prime number, Definition \ref{BOXmullineux} of the Mullineux map is a combinatorial algorithm of calculating the tensor product of an irreducible representation of the symmetric group with the sign representation in characteristic $b$ given in the Introduction. This result was conjectured by Mullineux in 1979 and later proved by Ford and Kleshchev in \cite{ford1997proof}.

There are two other equivalent combinatorial ways to define the Mullineux map, see \cite{ford1997proof,Dimakis2018} for details. It is easy to see Mullineux map is an involution from the representation-theoretic definition, but not so obvious from Definition \ref{BOXmullineux}. 
\end{Remark}

\begin{Example}\label{example28}
The $3$-rims of $(7,5,1,1)$ and $(7,2,1)$ are labelled by integers $1,2,3$ in Figure \ref{fig28}. The partition $(7,5,1,1)$ has two $3$-segments and one $3'$-segment, and $(7,2,1)$ has two $3$-segments and no $3'$-segment. Their truncated $3$-rims are shaded respectively in Figure \ref{fig28}. Thus, $(7,5,1,1)^{\J_3} = (5,3,1)$ and $(7,2,1)^{\J_3} = (5,1,1)$.
\end{Example}

\begin{center}
\begin{figure}[h]
\begin{tikzpicture}[scale=0.5]
\draw[pattern=north west lines, pattern color=black] (0,0) rectangle (1,1);
\draw[pattern=north west lines, pattern color=black] (3,2) rectangle (5,3);
\draw[pattern=north west lines, pattern color=black] (5,3) rectangle (7,4);

\draw (0,0) -- (0,4);
\draw (1,0) -- (1,4);
\draw (2,2) -- (2,4);
\draw (3,2) -- (3,4);
\draw (4,2) -- (4,4);
\draw (5,2) -- (5,4);
\draw (6,3) -- (6,4);
\draw (7,3) -- (7,4);
\draw (0,4) -- (7,4);
\draw (0,3) -- (7,3);
\draw (0,2) -- (5,2);
\draw (0,1) -- (1,1);
\draw (0,0) -- (1,0);

\foreach \x/\y/\m in {+0.5/+0.5/$2$,+0.5/+1.5/$1$,+2.5/+2.5/$3$,+3.5/+2.5/$2$,+4.5/+2.5/$1$,+4.5/+3.5/$3$,5.5/+3.5/$2$,6.5/+3.5/$1$} 
    \node at (\x,\y) {\m};

\draw[pattern=north west lines, pattern color=black] (13,2) rectangle (14,3);
\draw[pattern=north west lines, pattern color=black] (17,3) rectangle (19,4);

\draw (12,1) -- (12,4);
\draw (13,1) -- (13,4);
\draw (14,2) -- (14,4);
\draw (15,3) -- (15,4);
\draw (16,3) -- (16,4);
\draw (17,3) -- (17,4);
\draw (18,3) -- (18,4);
\draw (19,3) -- (19,4);
\draw (12,1) -- (13,1);
\draw (12,2) -- (14,2);
\draw (12,3) -- (19,3);
\draw (12,4) -- (19,4);

\foreach \x/\y/\m in {+12.5/+1.5/$3$,+12.5/+2.5/$2$,+13.5/+2.5/$1$,+16.5/+3.5/$3$,+17.5/+3.5/$2$,+18.5/+3.5/$1$} 
    \node at (\x,\y) {\m};

\end{tikzpicture}
\caption{The truncated $3$-rims for $(7,5,1,1)$ and $(7,2,1)$.}
\label{fig28}
\end{figure}
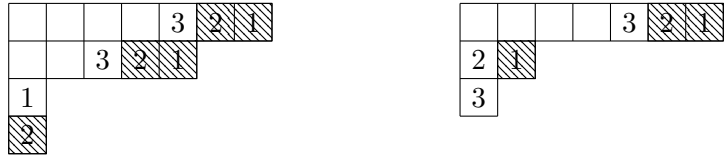
\end{center}

In the proof of the main theorem, we need to fully characterize the shape of the $b$-rim of certain partitions. The concept of rectangular decomposition allows for a simple description.

\begin{Definition}\label{definition of rectangular decomposition}
Given a partition $\lambda$, we label the boxes in its $b$-rim with positive integers from 1 to $N$ in order from northeast to southwest ($N$ is the number of boxes in the $b$-rim). Suppose $N=bk+r$ where $k\in \mathbb{N}$ and $0\leq r<b$. The \textit{$b$-rectangular decomposition} associated to the $b$-rim is the sequence of rectangles (of boxes) $r_1,\ldots, r_{k}$  such that each $r_i$ is the smallest rectangle containing the consecutive boxes labelled by $b(i-1)+1,\ldots,\min\{bi,bk+r\}$.

We let $r_i^x$ (resp. $r_i^y$) be the dimension of the rectangle $r_i$ in the north-south (resp. west-east) direction. Informally, we may refer to $r_i^x$ as the width of the rectangle and $r_i^y$ as the length of the rectangle.
\end{Definition}

\begin{Remark}
More specifically, the ``smallest" rectangle containing a set $S$ of boxes is the intersection of all rectangles parallel to the axes that contain $S$. So the smallest rectangle that contains $b$-consecutive boxes on the $b$-rim necessarily satisfies $r_i^x + r_i^y = b+1.$ In the context of Definition \ref{definition of rectangular decomposition}, for every $i$ except possibly the final one (when $\lambda$ contains a $b'$-segment), we have $r_i^x + r_i^y = b+1$. 
\end{Remark}

\begin{Example}
The $5$-rectangular decomposition of $(12,9,9,7,5,2,2,1)$ is the sequence of rectangles $r_1$, $r_2$, $r_3$, $r_4$ as shown in Figure \ref{fig210}. Note that $r_i^x=2$ for $i=1,2,3,4$, and $r_j^y=4$ for $j=1,2,3$ and $r_4^y=2$. Moreover, boxes from 1 to 10 form a $b$-segment, and boxes from 11 to 18 form a $b'$-segment. Also note that by construction, the rectangles may have at most one overlap in the columns, but they occupy all rows of $\lambda$ without overlapping.
\begin{center}
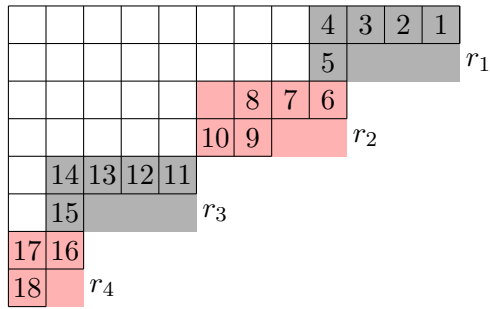
\begin{figure}[h]
\begin{tikzpicture}[scale=0.5]
\fill[black!30!white] (8,7) rectangle (12,9);
\fill[red!30!white] (5,5) rectangle (9,7);
\fill[black!30!white] (1,3) rectangle (5,5);
\fill[red!30!white] (0,1) rectangle (2,3);
\draw (0,1) -- (0,9);
\draw (1,1) -- (1,9);
\draw (2,2) -- (2,9);
\draw (3,4) -- (3,9);
\draw (4,4) -- (4,9);
\draw (5,4) -- (5,9);
\draw (6,5) -- (6,9);
\draw (7,5) -- (7,9);
\draw (8,6) -- (8,9);
\draw (9,6) -- (9,9);
\draw (10,8) -- (10,9);
\draw (11,8) -- (11,9);
\draw (12,8) -- (12,9);
\draw (12,8) -- (12,9);
\draw (0,1) -- (1,1);
\draw (0,2) -- (2,2);
\draw (0,3) -- (2,3);
\draw (0,4) -- (5,4);
\draw (0,5) -- (7,5);
\draw (0,6) -- (9,6);
\draw (0,7) -- (9,7);
\draw (0,8) -- (12,8);
\draw (0,9) -- (12,9);
\foreach \x/\y/\m in {+12.5/+7.5/$r_1$,+9.5/+5.5/$r_2$,+5.5/+3.5/$r_3$,+2.5/+1.5/$r_4$, 11.5/8.5/$1$, 10.5/8.5/$2$, 9.5/8.5/$3$, 8.5/8.5/$4$, 8.5/7.5/$5$, 8.5/6.5/$6$,7.5/6.5/$7$, 6.5/6.5/$8$, 6.5/5.5/$9$, 5.5/5.5/$10$,4.5/4.5/$11$,3.5/4.5/$12$, 2.5/4.5/$13$,1.5/4.5/$14$, 1.5/3.5/$15$, 1.5/2.5/$16$,0.5/2.5/$17$, 0.5/1.5/$18$}
\node at (\x,\y) {\m};

\end{tikzpicture}

\caption{The $5$-rectangular decomposition of $(12,9,9,7,5,2,2,1)$.}
\label{fig210}
\end{figure}
\end{center}
\end{Example}

\subsection{Regularization and Column Regularization}
At the beginning of Section 2, we embed the Young diagram of a partition into the coordinate plane and refer to a box and its southeast vertex by the same name $(i,j)$. From now on, any unit square with vertices in $\mathbb{Z}^2$ will also be called a \textit{box} and identified with its southeast vertex, i.e., when we refer to a box $(i,j)\in \mathbb{Z}^2$, we mean the unit square whose southeast vertex has coordinates $(i,j)$.

\begin{Definition}
\label{ladders}
Given two nonnegative integers $a<b$ and a partition $\lambda$, we define the following:
\begin{enumerate}
\item For any box $A=(i_0,j_0)\in\mathbb{Z}^2$, the \textit{ladder} passing through $A$ is defined to be the finite collection of boxes:
$$\ell_A=\ell_{i_0,j_0}\coloneqq\{(i,j)\in\mathbb{N}_{>0}^2:\exists\ t\in \mathbb{Z}, i = i_0-ta, j = j_0 + t(b-a)\}.$$
\item We define the \textit{column regularization} $\lambda^{\Cr_{a,b}}$ of $\lambda$ through the following procedure. For each ladder $\ell_{A}:A\in\mathbb{Z}^2$, if $\lambda\cap \ell_A\neq\emptyset$, slide the boxes in the intersection southwards along the ladder $\ell_A$ to the $|\lambda\cap \ell_A|$ south-most boxes of $\ell_A$. The resulting set of boxes is $\lambda^{\Cr_{a,b}}$, which is not necessarily a partition.
\item For any box $A=(i_0,j_0)\in\mathbb{Z}^2$, the \textit{dual ladder} passing through $A$ is defined to be:
$$\hat{\ell}_A=\hat{\ell}_{i_0,j_0}\coloneqq\{(i,j)\in\mathbb{N}_{>0}^2:\exists\ t\in \mathbb{Z}, i = i_0-t(b-a), j = j_0 + ta\}.$$
\item The \textit{regularization} $\lambda^{\Reg_{a,b}}$ of $\lambda$ is defined to be $\lambda^{\Tr\Cr_{a,b}\Tr}$. Equivalently, $\lambda^{\Reg_{a,b}}$ is the resulting partition after sliding boxes in each nonempty $\lambda\cap \hat{\ell}_A:A\in\mathbb{Z}^2$ northwards along the dual ladders $\hat{\ell}_A$ to the $|\lambda\cap \hat{\ell}_A|$ north-most boxes of $\hat{\ell_A}$.\end{enumerate}
We say that $\lambda$ is \textit{$\Cr_{a,b}$-valid} (resp. \textit{$\Reg_{a,b}$-valid}) if $\lambda^{\Cr_{a,b}}\in\mathcal{P}$ (resp. $\lambda^{\reg_{a,b}}\in\mathcal{P}$). 

\end{Definition}

\begin{Remark}
\begin{enumerate}
\item Note that though $A=(i_0,j_0)$ might not be in $\lambda$ or even not in the first quadrant, $\ell_A$ (resp. $\hat{\ell}_A$) only denotes some of the boxes in the first quadrant on the ``underlying line": $y-j_0+\frac{b-a}{a}(x-i_0)=0$ (resp. $y-j_0+\frac{a}{b-a}(x-i_0)=0$). 
\item Different subscripts can be used to indicate the same (resp. dual) ladder, for example when $a=2,b=5$, $\ell_{1,9}=\ell_{-1,12}=\ell_{3,6}$. But for the sake of clarity, in all the figures in this paper, we will only draw out the ``underlying line" when referring to a (resp. dual) ladder.
\item Only when $a$ and $b$ are co-prime, the (resp. dual) ladder coincide with the set of all positive integer points (boxes) on its ``underlying line."
\item All the boxes on a (resp. dual) ladder have the same $b$-residue, we say this residue to be the residue of the (resp. dual) ladder, denoted by $\res_b\ell$ (resp. $\res_b\hat\ell$).
\end{enumerate}
\end{Remark}

\begin{Example}\label{example colreg}
The partition $(3,2,2,1)$ is $\Cr_{2,3}$-valid and $(3,2,2,1)^{\Cr_{2,3}}=(2,2,2,1,1)$, but $\Cr_{2,3}$ maps $(3,2,2)$ to $(2,1,2,1,1) \notin \mathcal{P}$, hence $(3,2,2)$ is not $\Cr_{2,3}$-valid, which is shown in Figure \ref{fig214}.
\end{Example}
\begin{center}
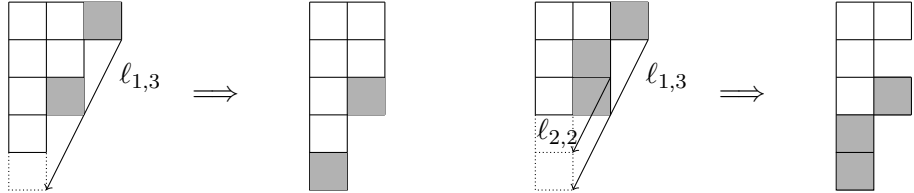
\begin{figure}[h]
\begin{tikzpicture}[scale=0.5]
\draw (1,0) -- (0,0) -- (0,4);
\draw (1,0) -- (1,4);
\draw (2,1) -- (2,4);
\draw (3,3) -- (3,4);
\draw (0,1) -- (2,1);
\draw (0,2) -- (2,2);
\draw (0,3) -- (3,3);
\draw (0,4) -- (3,4);
\draw[densely dotted] (0,0) -- (0,-1) -- (1,-1) -- (1,0);
\draw[->] (3,3) -- (1,-1);
\draw (9,0) -- (8,0) -- (8,4);
\draw (9,0) -- (9,4);
\draw (10,1) -- (10,4);
\draw (8,1) -- (10,1);
\draw (8,2) -- (10,2);
\draw (8,3) -- (10,3);
\draw (8,4) -- (10,4);
\draw (8,0) -- (8,-1) -- (9,-1) -- (9,0);
\fill[black!30!white] (1.02,1.02) rectangle (1.98,1.98);
\fill[black!30!white] (2.02,3.02) rectangle (2.98,3.98);
\fill[black!30!white] (9.02,1.02) rectangle (9.98,1.98);
\fill[black!30!white] (8.02,-0.98) rectangle (8.98,-0.02);
\foreach \x/\y/\m in {5.5/1.5/$\Longrightarrow$,19.5/1.5/$\Longrightarrow$,3.5/2/$\ell_{1,3}$,17.5/2/$\ell_{1,3}$,14.6/0.5/$\ell_{2,2}$} 
    \node at (\x,\y) {\m};

\draw (14,1) -- (14,4);
\draw (15,1) -- (15,4);
\draw (16,1) -- (16,4);
\draw (17,3) -- (17,4);
\draw (14,1) -- (16,1);
\draw (14,2) -- (16,2);
\draw (14,3) -- (17,3);
\draw (14,4) -- (17,4);
\draw[densely dotted] (14,1) -- (14,-1) -- (15,-1) -- (15,1);
\draw[densely dotted] (14,0) -- (15,0);
\fill[black!30!white] (15.02,1.98) rectangle (15.98,1.02);
\fill[black!30!white] (16.02,3.98) rectangle (16.98,3.02);
\fill[black!30!white] (23.02,0.98) rectangle (23.98,2.02);
\fill[black!30!white] (22.02,-1.02) rectangle (22.98,0.02);
\fill[black!30!white] (15.02,2.98) rectangle (15.98,2.02);
\fill[black!30!white] (22.02,-0.02) rectangle (22.98,1.02);
\draw[->] (17,3) -- (15,-1);
\draw[->,black] (16,2) -- (15,0);

\draw (23,0) -- (22,0) -- (22,4);
\draw (23,0) -- (23,4);
\draw (24,1) -- (24,2);
\draw (24,3) -- (24,4);
\draw (22,1) -- (24,1);
\draw (22,2) -- (24,2);
\draw (22,3) -- (24,3);
\draw (22,4) -- (24,4);
\draw (22,0) -- (22,-1) -- (23,-1) -- (23,0);

\end{tikzpicture}
\caption{An example of a $\Cr_{2,3}$-valid partition $(3,2,2,1)$ and an $\Cr_{2,3}$-invalid one $(3,2,2)$.}
\label{fig214}
\end{figure}
\end{center}

By definition, $\lambda$ is $\Reg_{a,b}$-valid exactly when $\lambda^{\Tr}$ is $\Cr_{a,b}$-valid, so we only specify the criteria for $\lambda$ being $\Cr_{a,b}$-valid. 

\begin{Lemma}\label{lemma for colreg valid}
Given $\lambda\in\mathcal{P}$ and positive integers $a<b$, $\lambda$ is $\Cr_{a,b}$-valid if and only if for all $y\in[1,\lambda_1]$, either $\ell_{1,y}\subset\lambda$, or there exists a box $(i,j)\in \lambda\cap \ell_{0,y}$ such that $(i+1,j)\notin\lambda$. In the first case where $\ell_{1,y}\subset\lambda$, we say that the ladder $\ell_{1,y}$ is ``full".
\end{Lemma}

\begin{proof}
$\lambda^{\Cr_{a,b}}\in \mathcal{P}$ is equivalent to the following:

For every $(i,j)\in \lambda^{\Cr_{a,b}}$, 
\begin{enumerate}
\item if $j\ge 2$, $(i,j-1)\in\lambda^{\Cr_{a,b}}$;
\item if $i\ge 2$, $(i-1,j)\in\lambda^{\Cr_{a,b}}$.
\end{enumerate}

(1) is always true for any $\lambda$. In fact for every $(x,y)\in \lambda$, $y\ge2$, consider the two ladders $\ell_{x,y}$ and $\ell_{x,y-1}$. If the ladder $\ell_{x,y}$ contains a box of the form $(i, 1)$, then $|\ell_{x,y-1}|=|\ell_{x,y}|-1$ and $|(\ell_{x,y-1}\cap\lambda)|\geq| (\ell_{x,y}\cap\lambda)|-1$ since (1) is satisfied by $\lambda$. If $\ell_{x,y}$ does not contain a box of the form $(i,1)$, then $|\ell_{x,y-1}|=|\ell_{x,y}|$ and $|(\ell_{x,y-1}\cap\lambda)|\geq| (\ell_{x,y}\cap\lambda)|$. Because $\Cr_{a,b}$ fixes the number of boxes on each ladder, and moves the boxes on each ladder to the south-most positions, for any $(i,j)\in \lambda^{\Cr_{a,b}}$ and $j\geq2$, $(i,j-1)\in\lambda^{\Cr_{a,b}}$.

To satisfy condition (2), for every $(x,y)\in \lambda$, $x\ge2$, we consider $\ell_{x,y}$ and $\ell_{x-1,y}$. Similar as case (1), if $\ell_{x,y}$ contains a box of the form $(1,j)$, then $|\ell_{x-1,y}|=|\ell_{x,y}|-1$ and $|(\ell_{x-1,y}\cap\lambda)|\geq| (\ell_{x,y}\cap\lambda)|-1$ since when $i\ge 2$ for each $(i,j)\in\lambda$, $(i-1,j)\in \lambda$. If $\ell_{x,y}$ does not contain a box of the form $(1,j)$, then $|\ell_{x-1,y}|=|\ell_{x,y}|$ and $|(\ell_{x-1,y}\cap\lambda)|\geq| (\ell_{x,y}\cap\lambda)|$. Hence $\lambda^{\Cr_{a,b}}\notin \mathcal{P}$ occurs only in the former case where $|(\ell_{x-1,y}\cap\lambda)|= |(\ell_{x,y}\cap\lambda)|-1$, $\ell_{x,y}$ contains through a box of the form $(1,j)$, and $\ell_{x,y}\not\subset\lambda$. In this case, without loss of generality, we suppose $x=1$ and $y\in [1,\lambda_1]$. Then the only situation when $\lambda^{\Cr_{a,b}}\notin\mathcal{P}$ is when $\ell_{1,y}\not\subset\lambda$, and pairs of boxes $(i,j)\in \ell_{1,y}$, $i\geq2$ and $(i-1,y)\in\ell_{0,y}$ are either both present in $\lambda$ or both not in $\lambda$ simultaneously. Under these conditions, the box $(1,y)\in \ell_{1,y}$ slides to a position where the box immediately north to it is not in $\lambda^{\Cr_{a,b}}$.
\end{proof}

\begin{Remark}
Referring to Example \ref{example colreg}, the box $(3,2)\in\ell_{0,3}\cap(3,2,2)^{\Cr_{2,3}}$ while $(2,2)\notin(3,2,2)^{\Cr_{2,3}}.$ In fact, in $(3,2,2)$, there is no box $(i,j)\in \ell_{0,3}\cap\lambda$ where $(i+1,j)\notin\lambda,$ so it is not $\Cr_{a,b}$-valid. 
\end{Remark}

It is important to note the following proposition, which in part illustrates the subtleties when a second parameter is added. In previous works, the notions of $\Reg$-valid and $\Cr$-valid were irrelevant since only one parameter was used. In fact, for any value of $b$ and any partition $\lambda$, the maps $\Reg_{b}$ and $\Cr_{b}$ would be well-defined operators on $\lambda.$

\begin{Proposition}
Every partition $\lambda$ is $\Reg_{1,b}$-valid and $\Cr_{1,b}$-valid for all $b$.
\end{Proposition}

\begin{proof}
To show that $\lambda$ is $\Cr_{1,b}$-valid, we directly apply Lemma \ref{lemma for colreg valid}. For all $y\in [1,\lambda_1]$, either $\ell_{1,y}\subset \lambda$ or there exists a box $(i+1,j)$ furthest north along $\ell_{1,y}$ such that $(i+1,j)\notin \lambda$. Then the box $(i,j + b-1)$, which is the box immediately northeast of $(i+1,j)$ along the ladder $\ell_{1,y}$, must be in $\lambda$. Therefore, $(i,j)\in\lambda\cap \ell_{0,y}$ while $(i+1,j)\notin \lambda$, so the condition in Lemma \ref{lemma for colreg valid} is satisfied.

We apply the same argument to all ladders $\ell_{1,y}$ for $y\in [1,\lambda_1]$ and conclude that $\lambda$ is $\Cr_{1,b}$-valid. By applying the proof to $\lambda^{\Tr},$ we see that $\lambda$ is $\Reg_{1,b}$-valid as well.
\end{proof}

\begin{Definition}\label{definition: b-regular and ab-regular}
Given positive integers $a<b$, a partition $\lambda$ is \textit{$(a,b)$-regular} if $\lambda^{\Reg_{a,b}}=\lambda.$ 
\end{Definition}
\begin{Remark}
$(1,b)$-regular is equivalent to the notion of $b$-regular. Also $(a,b)$-regular implies $b$-regular for any positive integer $a$ with $a<b$.
\end{Remark}

For positive integers $a<b$, we wish to define an intermediate operator called the \textit{column semi-regularization} $\Sr_{a,b}$ to compute $\Cr_{a,b}$ in a row-by-row fashion (parallel to that of the Mullineux transpose as in Equation \eqref{eq: M recursion}): $\lambda^{\Cr_{a,b}}_i = |\lambda^{\Sr_{a,b}^{i-1}}|-|\lambda^{\Sr_{a,b}^{i}}|$, or recursively, 

\begin{equation}\label{eq: Cr recursion}
    \lambda^{\Cr_{a,b}} = (|\lambda|-|\lambda^{\Sr_{a,b}}|)\oplus \lambda^{\Sr_{a,b}\Cr_{a,b}}.
\end{equation}
The only way to do this is the following: 
\begin{Definition}
Given a $\Cr_{a,b}$-valid partition $\lambda$, the \textit{column semi-regularization} $\lambda^{\Sr_{a,b}}$ of $\lambda$ is defined by the following procedure. For all $(1,y)\in\lambda$, if $\ell_{1,y}\not\subset\lambda$, slide $(1,y)$ to the north-most position $(i,j)$ in $\ell_{1,y}\setminus\lambda$ such that $(i-1,j)\in\lambda$. Then, we remove what remains of the first row, and the resulting partition is defined to be $\lambda^{\Sr_{a,b}}$.
\end{Definition}

It is clear that $\Sr_{a,b}$ is well-defined by Lemma \ref{lemma for colreg valid}, and $\lambda^{\Sr_{a,b}}$ is $\Cr_{a,b}$-valid if $\lambda$ is. Note that we only perform $\Sr_{a,b}$ to $\Cr_{a,b}$-valid partitions.

\begin{Example}
In the diagram shown in Figure \ref{fig220}, the boxes $(1,12)$ and $(1,13)$ shift under $\Sr_{2,5}$ to $(3,10)$ and $(5,6)$ respectively. We have $$(13,10,9,7,5,2,2,1)^{\Sr_{2,5}} = (10,10,7,6,2,2,1).$$
\end{Example}

\begin{center}
\begin{figure}[h]
\begin{tikzpicture}[scale=0.5]

\draw[densely dotted] (9,7) -- (10,7) -- (10,6) -- (9,6);
\draw[densely dotted] (5,5) -- (6,5) -- (6,4) -- (5,4);
\fill[black!30!white] (11.02,9.02) rectangle (12.98,7.98);
\draw[->] (13,8) -- (10,6);
\draw[->,black] (12,8) -- (6,4);
\draw[pattern=north west lines, pattern color=black] (0,9) rectangle (11,8);

\draw (0,1) -- (0,9);
\draw (1,1) -- (1,9);
\draw (2,2) -- (2,9);
\draw (3,4) -- (3,9);
\draw (4,4) -- (4,9);
\draw (5,4) -- (5,9);
\draw (6,5) -- (6,9);
\draw (7,5) -- (7,9);
\draw (8,6) -- (8,9);
\draw (9,6) -- (9,9);
\draw (10,7) -- (10,9);
\draw (11,8) -- (11,9);
\draw (12,8) -- (12,9);
\draw (13,8) -- (13,9);
\draw (0,1) -- (1,1);
\draw (0,2) -- (2,2);
\draw (0,3) -- (2,3);
\draw (0,4) -- (5,4);
\draw (0,5) -- (7,5);
\draw (0,6) -- (9,6);
\draw (0,7) -- (10,7);
\draw (0,8) -- (13,8);
\draw (0,9) -- (13,9);
\foreach \x/\y/\m in {13/7/$\ell_{1,13}$,9/5/$\ell_{1,12}$} 
    \node at (\x,\y) {\m};

\end{tikzpicture}
\caption{An example of column semi-regularization ${\Sr_{2,5}} $: \\$(13,10,9,7,5,2,2,1)^{\Sr_{2,5}} = (10,10,7,6,2,2,1).$}
\label{fig220}
\end{figure}
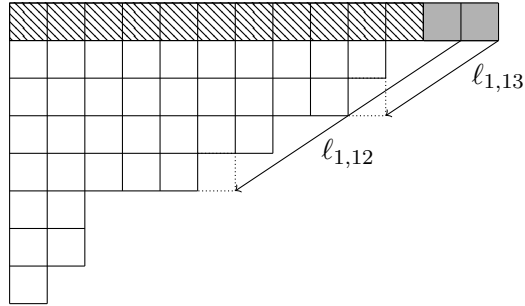
\end{center}

\subsection{Cores and Quotients} Cores and quotients are crucial concepts of partitions, and many attempts have been made to understand the Mullineux map in terms of cores and quotients. We first recall the definition of cores and refer to \cite{macdonald1998symmetric} for that of quotients and their basic properties.

\begin{Definition}
\label{haimancore}
The \textit{$b\con$core} $\Core_b(\lambda)$ of any partition $\lambda$ is the partition that remains after removing as many $b\con$ribbons in succession as possible. It is well-known that the result is independent of the choice of removals. If $\lambda=\Core_b(\lambda)$, then $\lambda$ itself is called a \textit{$b$-core}.
\end{Definition}

%\begin{Definition}
%\label{haimanquotient}
%Given a partition $\lambda$, for each residue class $k$, the collection of boxes $A\in\lambda$ satisfying $b\mid H_A$, $\res_b\mathfrak{h}_A=k$, $\res_b\mathfrak{f}_A=k+1$ forms an ``exploded" copy of a partition which we denote as $\lambda^{(k)}$. The \textit{$b$-quotient} of a partition $\lambda$ is defined to be the $b\con$tuple of partitions, $\Quot_b(\lambda) =(\lambda^{(0)},\lambda^{(1)},...,\lambda^{(b-1)})$. The \textit{$b$-weight} of $\lambda$ is defined to be $|\lambda|_b=\sum_{k=0}^{b-1}|\lambda^{(k)}|$.
%\end{Definition}

\begin{Definition}
The $b$-content of the partition $\lambda$ is a tuple $(c_0,\ldots,c_{b-1})$ where $c_i$ is the number of boxes in $\lambda$ with residue $i$.
\end{Definition}

\begin{Proposition}[\cite{macdonald1998symmetric}]
The $b$-content determines the $b$-core of a partition.
\end{Proposition}

Since the residue of boxes on each ladder is the same, the $b$-content is fixed when boxes move within the same ladders. Thus, $b$-core is an invariant under the (column) regularization process, and we have the following result:
\begin{Lemma}
\label{reg preserves core}
For a $\Reg_{a,b}$-valid (resp. $\Cr_{a,b}$-valid) partition $\lambda$, we have
$$\Core_b(\lambda^{\Reg_{a,b}})=\Core_b(\lambda)\quad(\text{resp. }\Core_b(\lambda^{\Cr_{a,b}})=\Core_b(\lambda))$$
In particular, we have $|\lambda^{\Reg_{a,b}}|_b=|\lambda|_b$ (resp. $|\lambda^{\Cr_{a,b}}|_b=|\lambda|_b$).
\end{Lemma}

\begin{Remark}
Mullineux transpose also fixes the core of a $b$-regular partition, as shown in \cite{ford1997proof}, which is compatible with Theorem \ref{theorem generalizing BOX}.
\end{Remark}

\section{Rectangular Decomposition and Proof to the Main Theorem}
In this section we will completely characterize the shape of $\lambda$ satisfying the conditions restraining hook shape from Equations \eqref{equation: shallow}. Using this description, we then conclude with a proof of Theorem \ref{theorem generalizing BOX}. It happens that the conditions from Equations \eqref{equation: shallow} give a strong result on the dimensions of any rectangle that contains $b$ consecutive boxes on the rim of $\lambda$. Although we will often apply the lemma below to rectangles in the rectangular decomposition of $\lambda$, it is proven in greater generality here. The following result gives a strong restriction on the dimensions $r^x$ and $r^y$ (see Definition \ref{definition of rectangular decomposition}) of the rectangles.

\begin{Lemma}\label{lemma prohibiting shallow after steep rectangle}
Given a partition $\lambda$, suppose $r_\alpha$ and $r_{\beta}$ are two smallest rectangles each containing $b$ consecutive boxes of the rim of $\lambda$, i.e. $r_{\alpha}^x+ r_{\alpha}^y = r_{\beta}^x + r_{\beta}^y = b+1$. Let the box $B_0 = (i_0,j_0)$ be the box furthest northeast in $r_{\alpha}$. Suppose the following conditions are satisfied:

\begin{enumerate}
\item $(i_0,j_0+1)\notin\lambda$;
\item $r_\alpha^x>a\ge r_{\beta}^x$;
\item $r_{\beta}$ contains a box strictly southwest of $r_\alpha$.
\end{enumerate}
Then $\lambda$ contains a hook with leg length $a$ and arm length $b-a-1.$ 
\end{Lemma}

Note that $r_\alpha$ and $r_\beta$ are not necessarily rectangles from the rectangular decomposition of $\lambda$; the indices $\alpha$ and $\beta$ are irrelevant to the statement of the Lemma but will be referenced later in the section when we apply Lemma \ref{lemma prohibiting shallow after steep rectangle} to rectangles in the rectangular decomposition. 

\begin{Example}\label{example for main lemma}
The example in Figure \ref{fig31lemma} illustrates the statement of Lemma \ref{lemma prohibiting shallow after steep rectangle}. The rectangles $r_\alpha$ and $r_\beta$ are  colored yellow and the corresponding boxes in the rim are outlined in red. We see that $B_0 = (i_0,j_0) = (1,8)$ and $(1,9)\notin \lambda$, thus the first condition is satisfied. We have $r_\alpha^x = 3$, $r_\alpha^y = 3$, $r_\beta^x = 2$, and $r_\beta^y = 4$, thus the second condition is satisfied. Finally, from the diagram, the third condition is satisfied as well. Lemma \ref{lemma prohibiting shallow after steep rectangle} implies that there is a hook of leg length $2$ and arm length $2$, which is outlined in blue.
\end{Example}

\begin{center}
\begin{figure}[h]
\begin{tikzpicture}[scale=0.5]
\draw (0,0) -- (0,6);
\draw (1,0) -- (1,6);
\draw (2,1) -- (2,6);
\draw (3,1) -- (3,6);
\draw (4,1) -- (4,6);
\draw (5,1) -- (5,6);
\draw (6,1) -- (6,6);
\draw (7,3) -- (7,6);
\draw (8,4) -- (8,6);
\draw (0,0) -- (1,0);
\draw (0,1) -- (6,1);
\draw (0,2) -- (6,2);
\draw (0,3) -- (7,3);
\draw (0,4) -- (8,4);
\draw (0,5) -- (8,5);
\draw (0,6) -- (8,6);
\draw[white] (0,-2)--(5,-2);
\draw[line width=1pt,  black] (0,0) -- (0,6) -- (8,6) --(8,4) --(7,4) --(7,3)--(6,3)--(6,1)--(1,1)--(1,0)--(0,0);
\draw[line width=1pt,red] (0,0)--(0,2)--(4,2)--(4,1)--(1,1)--(1,0)--(0,0);
\draw[line width=1pt,red] (5,3) --(7,3)--(7,4)--(8,4)--(8,6)--(7,6)--(7,5)--(6,5)--(6,4)--(5,4)--(5,3);
\draw [fill=yellow, opacity=0.2]
       (5,6) -- (8,6) -- (8,3) -- (5,3) -- cycle;
\draw [fill=yellow, opacity=0.2]
       (0,0) -- (0,2) -- (4,2) -- (4,0) -- cycle;     
\draw[line width=1.5pt,  blue] (4,4) -- (7,4) --(7,3) --(5,3)--(5,1) --(4,1)--(4,4);
\node at (7.5,5.5) {$B_0$};
\node at (8.7,4) {$r_\alpha$};
\node at (4.7,0) {$r_\beta$};
\end{tikzpicture}
\caption{Illustration of the statement of Lemma \ref{lemma prohibiting shallow after steep rectangle} on $(8,8,7,6,6,1)$ with $a=2$ and $b=5$.}
\label{fig31lemma}
\end{figure}
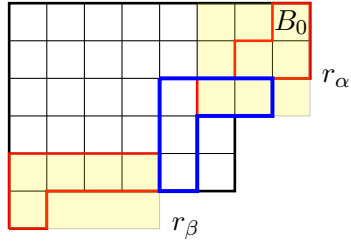
\end{center}

\begin{proof}[Proof of Lemma \ref{lemma prohibiting shallow after steep rectangle}]
Let $B_c = (i_c,j_c)$ be the box furthest southwest in $r_{\beta}$. We must have $j_c<j_0$ and $i_c>i_0$ since $r_\beta$ contains a box strictly southwest of $r_\alpha$. Denote by $\mathcal{B}$ the set of consecutive boxes in the rim of $\lambda$ starting at $B_0$ and ending at $B_c$. We define
\begin{align*}
\mathcal{S}= & \{(i,j)\in \mathcal{B} : (i+1,j)\notin \lambda\}\\\mathcal{E}=  & \{(i+a,j-(b-a-1)) : (i,j)\in \mathcal{B}, (i,j+1)\notin\lambda\}.
\end{align*}

For an example of this definition, see Figure \ref{fig332}.

We first show that $\mathcal{S}\cap\mathcal{E}\neq \emptyset$, then we use a box in this intersection to find a corresponding hook of leg length $a$ and arm length $b-a-1$. We prove that $\mathcal{S}\cap\mathcal{E}\neq \emptyset$ by examining the sequence of boxes in $\mathcal{S}$ and $\mathcal{E}$ when written from the northeast to southwest direction. Our approach is to show that since $\mathcal{S}$ begins south of $\mathcal{E}$ and ends north of $\mathcal{E}$, an ``intermediate value" argument implies that the two sets must intersect.

Since there is at least one box in the rim for each $y$-coordinate and $x$-coordinate, the set $\mathcal{S}$ contains exactly one box for every $y$-coordinate in $[j_c+1,j_0]$, and $\mathcal{E}$ contains one box for every $x$-coordinate in $[i_0+a,i_c+a].$ We sort the boxes in $\mathcal{S}$ and $\mathcal{E}$ in order from northeast to southwest and consider the two sequences. Since $B_0 = (i_0,j_0)\in \lambda$ and $(i_0,j_0+1)\notin\lambda$ by the first condition, the box $(i_0+a,j_0-(b-a-1))$ must lie in $\mathcal{E},$ and we denote this box, the first box of $\mathcal{E}$, by $E_{i_0+a}=(i_0+a,j_0-(b-a-1))$. Denote the southwest-most box of $r_\alpha$ by $B_{b-1}=(i_{b-1},j_{b-1}) = (i_0+r_\alpha^x-1,j_0-r_\alpha^y+1)$. We know that $E_{i_0+a}$ is northwest to $B_b$ (not necessarily strict) and these two boxes determine a line of slope $1$. Since $j_c<j_0-(b-a-1)\leq j_{b-1}=j_0-r_\alpha^y+1$, the unique box $(i',j_0-(b-a-1))\in\mathcal{S}$ in the same column of $E_{i_0+a}$ is not north of $E_{i_0+a}$, i.e. $i'\geq i_0+a$.

Next, we consider the unique box $E_{i_c}=(i_c,j')$ in $\mathcal{E}$, which has $x$-coordinate $i_c$. Since $r_\beta^y\geq b-a+1$, we find that $j'>j_c$. Next, we find the unique box $(i'',j')\in\mathcal{S}$ in the same column of $E_{i_c}$. Then $i''\leq i_c$ since $(i'',j') \in \mathcal{S} \subset \lambda$. This argument is illustrated by Figure \ref{fig31pf}.

\begin{center}
\begin{figure}[h]
\begin{tikzpicture}[scale=0.5]
\draw (0,0) -- (0,4);
\draw (1,0) -- (1,4);
\draw (2,1) -- (2,4);
\draw (3,1) -- (3,4);
\draw (4,1) -- (4,4);
\draw (5,1) -- (5,4);
\draw (6,1) -- (6,4);
\draw (7,3) -- (7,4);
\draw (0,0) -- (1,0);
\draw (0,1) -- (6,1);
\draw (0,2) -- (6,2);
\draw (0,3) -- (7,3);
\draw (0,4) -- (7,4);
\node at (0.5,0.5) {$B_c$};
\node at (2.5,-0.5) {\tiny{$E_{i_c} = (i_c , j')$}};
\node at (5.5,-0.5) {\tiny{$(i'',j')$}};
\draw[->,black] (2.5,-0.2) -- (3.5,0.5);
\draw[->,black] (5.5,-0.2) -- (3.5,1.5);
\draw[white] (0,-2)--(5,-2);
\draw [densely dotted] (3,0)--(3,1)--(4,1)--(4,0)--(3,0);
\draw [fill=yellow, opacity=0.2]
       (0,0) -- (0,2) -- (4,2) -- (4,0) -- cycle;     
\end{tikzpicture}
\caption{Illustration of $E_{i_c} = (i_c , j')$ and a box $(i'',j')$ in $\mathcal{S}$ north of it. The rectangle $r_\beta$ is colored in yellow.}
\label{fig31pf}
\end{figure}
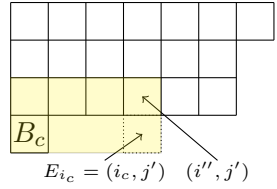
\end{center}

We now consider boxes in $\mathcal{S}$ and $\mathcal{E}$ inside the big rectangle (with sides parallel to the axis) whose northeast-most box is $E_{i_0+a}$ and southwest-most box is $E_{i_c}$. Boxes in $\mathcal{E}$ divide the rows into:
$$x_0=i_0+a<x_1<\ldots<x_k=i_c+1$$
such that the boxes in $\mathcal{E}$ on rows $[x_{j-1},x_j-1] $ lie in the same column $\varphi(j)$ and $\varphi(j)<\varphi(j-1)$ for all $j=1,\ldots,k$ (see Figure \ref{fig31}). Next, we define the function $\psi$ such that $(\psi(j),\varphi(j))$ is the unique box in $\mathcal{S}$ and column $\varphi(j)$. 

Suppose on the contrary that $\mathcal{S}\cap\mathcal{E}=\emptyset$. The above argument about $\mathcal{S}$ and $\mathcal{E}$ gives the conditions $\psi(1)\geq x_1$ and $\psi(k)< x_{k}$. Also note that $\psi(j)\geq \psi(j-1)$ because $\varphi(j)<\varphi(j-1)$ and we are considering the boxes of $\mathcal{S}$ from northeast to southwest. We show by induction that $\psi(j)\geq x_j$ for all $j$; this will give the necessary contradiction since we have already shown $\psi(k)< x$. 

The base case, where $j = 1$, is true from our reasoning above. Suppose that  $\psi(j)\geq x_j$ holds for indices strictly less than a fixed $j$. So $$\psi(j)\geq\psi(j-1)\geq x_{j-1}.$$

We know that $(x, \varphi(j))\in\mathcal{E}$ for $x\in[x_{j-1},x_j-1]$. Since $\mathcal{S} \cap \mathcal{E} = \emptyset$ by assumption, $\psi(j) \ge x_j$, which completes the induction. However, this implies that $\psi(k)\ge x_k = i_c+1$, which is a contradiction, so we must have $\mathcal{S}\cap\mathcal{E}\neq \emptyset.$ Figure \ref{fig31} illustrates this part of the proof.
 
\begin{center}
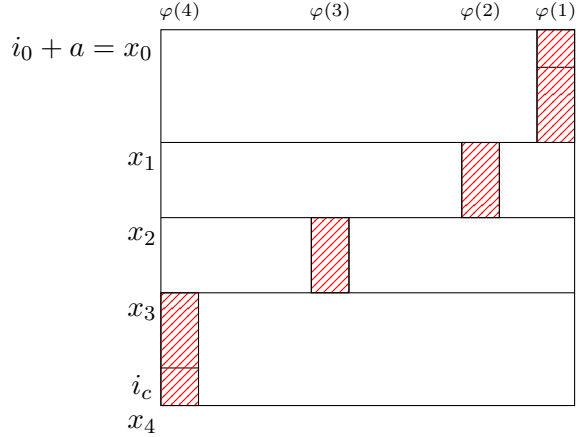
\begin{figure}[h]
\begin{tikzpicture}[scale=0.5]
\draw (0,0)--(0,10)--(11,10)--(11,0)--(0,0);
\draw (0,3)--(11,3);
\draw (0,5)--(11,5);
\draw (0,7)--(11,7);
\draw (1,0)--(1,3);
\draw(4,3)--(4,5);
\draw(5,3)--(5,5);
\draw(8,5)--(8,7);
\draw(9,5)--(9,7);
\draw(10,7)--(10,10);
\draw(0,1)--(1,1);
\draw (10,9)--(11,9);
\draw[pattern=north east lines, pattern color=red] (0,0) rectangle (1,3);
\draw[pattern=north east lines, pattern color=red] (4,3) rectangle (5,5);
\draw[pattern=north east lines, pattern color=red] (8,5) rectangle (9,7);
\draw[pattern=north east lines, pattern color=red] (10,7) rectangle (11,10);
\foreach \x/\y/\m in {-0.5/-0.5/$x_4$, -0.5/2.5/$x_3$, -0.5/4.5/$x_2$,-0.5/6.5/$x_1$,-2.1/9.5/$i_0+a=x_0$, -0.5/0.5/$i_c$,0.5/10.5/\tiny{$\varphi(4)$},4.5/10.5/\tiny{$\varphi(3)$},8.5/10.5/\tiny{$\varphi(2)$},10.5/10.5/\tiny{$\varphi(1)$}} 
\node at (\x,\y) {\m};
\end{tikzpicture}
\caption{Illustration of the shape of $\mathcal{S}$ and $\mathcal{E}$ in the proof of Lemma \ref{lemma prohibiting shallow after steep rectangle}. Boxes shaded black are in $\mathcal{S}$. Boxes shaded in red are in $\mathcal{E}$.}
\label{fig31}
\end{figure}
\end{center}

Pick $(i,j)\in \mathcal{S}\cap\mathcal{E}.$ Then $(i,j)\in \lambda$, and $(i+1,j)\notin \lambda$. Also, $(i-a,j+(b-a-1))\in\lambda$, and $(i-a,j+b-a)\notin\lambda$. Hence, $H_{i-a,j}$ has corresponding leg length $a$ and arm length $b-a-1$, as desired. See Figure \ref{fig332} for an illustration of this argument.
\end{proof}

\begin{center}
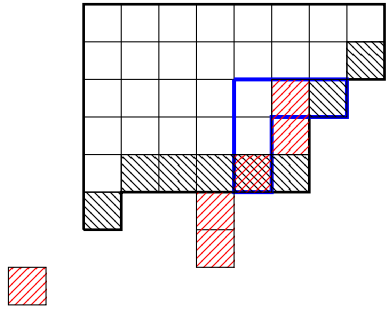
\begin{figure}[h]
\begin{tikzpicture}[scale=0.5]
\draw (0,0) -- (0,6);
\draw (1,0) -- (1,6);
\draw (2,1) -- (2,6);
\draw (3,1) -- (3,6);
\draw (4,1) -- (4,6);
\draw (5,1) -- (5,6);
\draw (6,1) -- (6,6);
\draw (7,3) -- (7,6);
\draw (8,4) -- (8,6);
\draw (0,0) -- (1,0);
\draw (0,1) -- (6,1);
\draw (0,2) -- (6,2);
\draw (0,3) -- (7,3);
\draw (0,4) -- (8,4);
\draw (0,5) -- (8,5);
\draw (0,6) -- (8,6);
\draw (3,0) -- (4,0);
\draw[line width=1pt,  black] (0,0) -- (0,6) -- (8,6) --(8,4) --(7,4) --(7,3)--(6,3)--(6,1)--(1,1)--(1,0)--(0,0);
\draw[line width=1.5pt,  blue] (4,4) -- (7,4) --(7,3) --(5,3)--(5,1) --(4,1)--(4,4);
% south ones
\draw[pattern=north west lines, pattern color=black] (7,5) rectangle (8,4);
\draw[pattern=north west lines, pattern color=black] (6,4) rectangle (7,3);
\draw[pattern=north west lines, pattern color=black] (1,2) rectangle (6,1);
\draw[pattern=north west lines, pattern color=black] (0,1) rectangle (1,0);
\draw[pattern=north east lines, pattern color=red] (5,2) rectangle (6,4);
\draw[pattern=north east lines, pattern color=red] (4,1) rectangle (5,2);
\draw[pattern=north east lines, pattern color=red] (3,-1) rectangle (4,1);
\draw[pattern=north east lines, pattern color=red] (-2,-2) rectangle (-1,-1);
\end{tikzpicture}

\caption{Illustration of the argument in the proof of Lemma \ref{lemma prohibiting shallow after steep rectangle} on $(8,8,7,6,6,1)$ with $a=2$ and $b=5$. The boxes in $\mathcal{S}$ are shaded with black lines and $\mathcal{E}$ is shaded with red lines. They intersect at $(5,5),$ which determines the hook $H_{3,5}$ with $a_{3,5} = b-a-1 = 2$ and $l_{3,5} = a = 2$, outlined in blue.}
\label{fig332}
\end{figure}
\end{center}

At this point, we note that there is a special type of hook of the shape
\begin{equation}\label{leg=ta}
    H_{i,j}=tb, \;l_{i,j}=ta \text{ for some }t\in\mathbb{N}_{>0}.
\end{equation}
Note that this type of hook is not $(a,b)$-shallow, and in the case of $b>2a$, it is not $(a,b)$-steep either.

\begin{Lemma}\label{lemma: ladders gain where empty}
Consider a $\Cr_{a,b}$-valid partition $\lambda$ containing no hook of shape \eqref{leg=ta}. For a ladder $\ell_{1,y} : y\in[1,\lambda_1]$ that is not full (i.e. $\ell_{1,y}\not \subset\lambda$), let $(i,j)$ be the box furthest north in $\ell_{1,y}\setminus\lambda$. Then, $(i-1,j)\in\lambda$.
\end{Lemma}

\begin{Example}
We illustrate the statement of Lemma \ref{lemma: ladders gain where empty} in Figure \ref{fig34} in the case of $a = 2$ and $b = 5$. On the left side, in $(8,5,4,1)$, the boxes $(5,1)$ and $(3,5)$ are the north-most boxes on ladders $\ell_{1,7}$ and $\ell_{1,8}$ respectively that are not contained in $\lambda$. The boxes immediately north of $(5,1)$ and $(3,5)$ are in $\lambda$ as Lemma \ref{lemma: ladders gain where empty} states.

On the right hand side, however, in $(8,4,4,2)$, the box $(3,5)$ is the north-most box on $\ell_{1,8}$ not contained in $\lambda$. In this case $(2,5)\notin\lambda$, which contradicts Lemma \ref{lemma: ladders gain where empty}. Then Lemma \ref{lemma for colreg valid} implies that $(4,2)\in\lambda$ and $(5,2)\notin\lambda$. This gives hook $H_{2,2}$ of shape \eqref{leg=ta}, which is not allowed by the statement of the lemma. 
\end{Example}
\begin{figure}[h]
\begin{tikzpicture}[scale=0.5]
\draw (0,0) -- (0,4);
\draw (1,0) -- (1,4);
\draw (2,1) -- (2,4);
\draw (3,1) -- (3,4);
\draw (4,1) -- (4,4);
\draw (5,2) -- (5,4);
\draw (6,3) -- (6,4);
\draw (7,3) -- (7,4);
\draw (8,3) -- (8,4);
\draw (0,0) -- (1,0);
\draw (0,1) -- (4,1);
\draw (0,2) -- (5,2);
\draw (0,3) -- (8,3);
\draw (0,4) -- (8,4);
\draw [densely dotted] (4,1)--(5,1)--(5,2)--(4,2);
\draw [densely dotted] (0,0)--(1,0)--(1,-1)--(0,-1)--(0,0);
\draw (8,3) -- (3.5,0) ;
\draw (7,3) -- (-0.5,-2) ;
\node at (5.5,0.5) {$\ell_{1,8}$};
\node at (-1,-1.5) {$\ell_{1,7}$};
\end{tikzpicture}
\quad
\quad
\quad
\begin{tikzpicture}[scale=0.5]
\draw (0,0) -- (0,4);
\draw (1,0) -- (1,4);
\draw (2,0) -- (2,4);
\draw (3,1) -- (3,4);
\draw (4,1) -- (4,4);
\draw (5,3) -- (5,4);
\draw (6,3) -- (6,4);
\draw (7,3) -- (7,4);
\draw (8,3) -- (8,4);
\draw (0,0) -- (2,0);
\draw (0,1) -- (4,1);
\draw (0,2) -- (4,2);
\draw (0,3) -- (8,3);
\draw (0,4) -- (8,4);
\draw (8,3) -- (0.5,-2) ;
\draw (7,3) -- (-0.5,-2) ;
\node at (5.5,0.5) {$\ell_{1,8}$};
\node at (-1,-1.5) {$\ell_{1,7}$};
\draw[line width=1.5pt,  blue] (1,3) -- (4,3) -- (4,2) -- (2,2) -- (2,0) -- (1,0) -- (1,3);
\draw [densely dotted] (0,0)--(1,0)--(1,-1)--(0,-1)--(0,0);
\draw [densely dotted] (1,0)--(2,0)--(2,-1)--(1,-1)--(1,0);
\end{tikzpicture}
\caption{Illustration of the statement of Lemma \ref{lemma: ladders gain where empty}. The ladders $\ell_{1,y}$ are labeled. A hook of shape \eqref{leg=ta} is outlined in blue. Key boxes in $\ell_{1,y}\setminus\lambda$ are drawn in dotted lines.}
\label{fig34}
\end{figure}
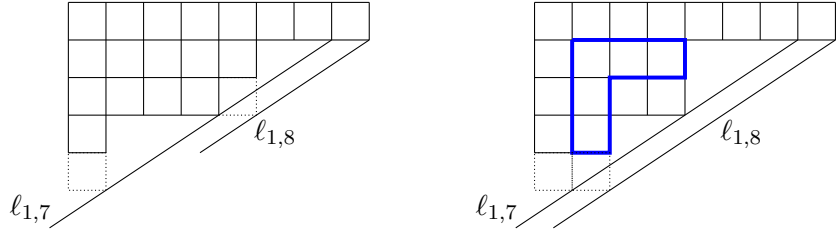

\begin{proof}[Proof of Lemma \ref{lemma: ladders gain where empty}]
Let $\ell_{1,y_1}$ be the westmost ladder that is not full, passing through a box in the first row. Then all $\ell_{1,y}$ east of $\ell_{1,y_1}$ are not full either. Thus, we denote by $\ell_{1,y_1},\ell_{1,y_1+1},\ldots,\ell_{1,\lambda_1}$ to be all the ladders passing through a box in the first row of $\lambda$ that are not fully contained in $\lambda$ (if there are any). We proceed by induction to show that the lemma holds for all of these ladders.

In the base case, we consider the ladder $\ell_{1,y_1}$. Let $(i,j)$ be the box furthest north on $\ell_{1,y_1}$ that is not in $\lambda$. Note that $i>1$ because $\ell_{1,y_1}$ contains a box in the first row. Suppose for the sake of contradiction that $(i-1,j)\notin\lambda$. Since $\lambda$ is $\Cr_{a,b}$-valid, Lemma \ref{lemma for colreg valid} implies $\exists\ t\in \mathbb{N}$ such that $(i+ta,j-t(b-a))\notin \lambda$ and $(i+ta-1,j-t(b-a))\in\lambda.$ By construction $\ell_{1,y_1-1}$ is completely contained in $\lambda$, so $(i,j-1)\in \ell_{1,y_1-1}\subset \lambda$, and thus $(i-1,j-1)\in \lambda$ as well. Then $H_{i-1,j-t(b-a)}$ satisfies $l_{i-1,j-t(b-a)} = ta$ and $a_{i-1,j-t(b-a)} = t(b-a)-1$, which contradicts the hypothesis that $\lambda$ contains no hook of shape \eqref{leg=ta}. Hence, our initial assumption was false, and $(i-1,j)\in\lambda$, which completes the base case. Figure \ref{fig33} illustrates this part of the proof.
\begin{center}
\begin{figure}[h]
\begin{tikzpicture}[scale=0.5]
\draw[line width=1pt]  (0,1)--(0,6)--(6,6)--(6,5)--(1,5)--(1,1)--(0,1);
\draw (0,2)--(1,2);
\draw (5,5)--(5,6);
\draw (5,5)--(5,4)--(6,4)--(6,5);
\draw(0.25,-0.5)--(9.25,5.5);
\draw [densely dotted] (0,1)--(0,0)--(1,0)--(1,1);
\draw [densely dotted] (6,4)--(7,4)--(7,5)--(6,5);
\draw [densely dotted] (7,5)--(7,6)--(6,6);
\draw[->,black] (0.5,6.5)--(0.5,5.5);
\draw[->,black] (7.5,3.5)--(6.5,4.5);
\draw[->,black] (-0.5,0.5)--(0.5,0.5);
\draw (0,5)--(1,5)--(1,6);
\foreach \x/\y/\m in {10/5.5/$\ell_{1,y_1}$} 
\node at (\x,\y) {\m};
\node at (-4.1,0.5) {$(i+ta,j-t(b-a))$};
\node at (0.5,6.9) {$(i-1,j-t(b-a))$};
\node at (8,3) {$(i,j)$};
\end{tikzpicture}
\caption{Illustration of the proof of Lemma \ref{lemma: ladders gain where empty}.}
\label{fig33}
\end{figure}
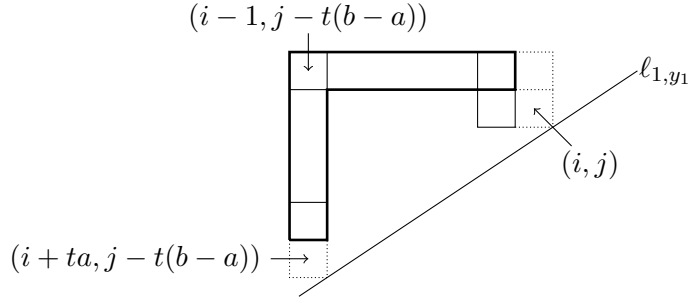
\end{center}

For the inductive step, we reason similarly as we did in the base case. Assume that the lemma holds for all $\ell_{1,y}$: $ y\in [y_1,y']$ for a fixed $y' < \lambda_1$. Let $(i,j)$ be the box furthest north on $\ell_{1,y'}$ that is not in $\lambda$. Once again, note that $i>1$ because $\ell_{1,y'}$ contains a box in the first row. Suppose for the sake of contradiction that $(i-1,j)\notin\lambda$.  Since all boxes in $\ell_{1,y'}$ north of $(i,j)$ are in $\lambda$, all boxes in $\ell_{1,y'-1}$ north of $(i,j)$ are also in $\lambda$. 

Consider the box $(i,j-1)$ on ladder $\ell_{1,y'}$. If $(i,j-1)\in\lambda$, then $(i-1,j-1)\in\lambda$ since it lies immediately north of $(i,j-1)$; if $(i,j-1)\notin\lambda$, then $(i,j-1)$ is the north-most box in $\ell_{1,y'}\setminus\lambda$. In this case, the inductive hypothesis also guarantees $(i-1,j-1)\in\lambda$. Thus, we conclude $(i-1,j-1) \in\lambda$ and $(i-1,j)\notin\lambda$ in all cases.

Using the exact same reasoning from the proof of the base case, we can use Lemma \ref{lemma for colreg valid} to find a box $(i+ta-1,j-t(b-a))\in \lambda$ such that $(i+ta,j-t(b-a))\notin \lambda$. This gives a hook $H_{i-1,j-t(b-a)}=tb$ with leg length $ta$ and contradicts the hypothesis that $\lambda$ contains no hook of shape \eqref{leg=ta}. Thus, we must have $(i-1,j)\in \lambda$, which completes the induction.
\end{proof}

Consider the $b$-rectangular decomposition of $\lambda$ in Definition \ref{definition of rectangular decomposition}. Define $\omega(\lambda)$ to be largest index such that 
$$r^x_{\omega(\lambda)}=a,\; r^y_{\omega(\lambda)} = b-a+1,$$ and define $\psi(\lambda)=\sum_{i=1}^{\omega(\lambda)}r_i^x.$ When no rectangle of such shape exists, we let $\omega(\lambda) = \psi(\lambda) = 0$.

\begin{Example}
When $a=2$ and $b=5$, $\omega((2,1,1)) = \psi((2,1,1)) = 0$; and $\omega((8,4,1)) = 2, \psi((8,4,1)) = 3$, which is shown in Figure \ref{fig32}.
\begin{center}
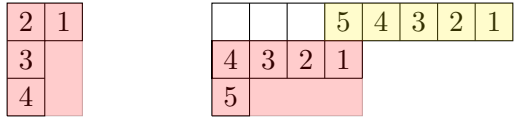
\begin{figure}[h]
\begin{tikzpicture}[scale=0.5]
\draw (0,0)--(1,0)--(1,2)--(2,2)--(2,3)--(0,3)--(0,0);
\draw (0,1)--(1,1);
\draw (0,2)--(1,2)--(1,3);
\foreach \x/\y/\m in {0.5/0.5/$4$,0.5/1.5/$3$,0.5/2.5/$2$,1.5/2.5/$1$} 
\node at (\x,\y) {\m};
\draw [fill=red, opacity=0.2]
       (0,0) rectangle (2,3);  
\end{tikzpicture}
\qquad
\qquad
\begin{tikzpicture}[scale=0.5]
\draw(0,0)--(0,3)--(8,3)--(8,2)--(0,2)--(0,1)--(4,1);
\draw(0,0)--(1,0)--(1,3);
\draw(2,1)--(2,3);
\draw(3,1)--(3,3);
\draw(4,1)--(4,3);
\draw(5,2)--(5,3);
\draw(6,2)--(6,3);
\draw(7,2)--(7,3);
\foreach \x/\y/\m in {0.5/0.5/$5$,0.5/1.5/$4$,1.5/1.5/$3$,2.5/1.5/$2$,3.5/1.5/$1$, 3.5/2.5/$5$,4.5/2.5/$4$,5.5/2.5/$3$,6.5/2.5/$2$,7.5/2.5/$1$} 
\node at (\x,\y) {\m};
\draw [fill=yellow, opacity=0.2]
       (3,2) rectangle (8,3);
\draw [fill=red, opacity=0.2]
       (0,0) rectangle (4,2);  
\end{tikzpicture}
\caption{Illustration of the definition of $\omega(\lambda)$ and $\psi(\lambda)$.}
\label{fig32}
\end{figure}
\end{center}
\end{Example}

We first give a description of the general shape of all $\lambda$ that $\Cr_{a,b}$-valid and contain no hooks given by \eqref{leg=ta}. Note that these conditions are strictly weaker than the conditions of Theorem \ref{theorem generalizing BOX}. Indeed, using only these two restrictions, we can already find a very specific description for the rectangular decomposition of $\lambda.$

\begin{Proposition}\label{proposition rectangular decomposition}
Given a $\Cr_{a,b}$-valid partition $\lambda$ containing no hooks of shape \eqref{leg=ta}, denote by $r_1,\ldots,r_k$ to be the rectangles in its $b$-rectangular decomposition, defined by Definition \ref{definition of rectangular decomposition},  with $r_\alpha^x+r_\alpha^y=b+1: \alpha\in[1, k]$. (Note that these are all the rectangles in the $b$-rectangular decomposition except for possibly the last one.) Then for all $\alpha\in [1,\omega(\lambda)]$, the widths of the rectangles satisfy $r^x_\alpha = a$, and for all $\alpha\in[\omega(\lambda)+1,k]$ the widths satisfy $r^x_{\alpha} >a$.
\end{Proposition}

\begin{Example}
We illustrate the statement of Proposition \ref{proposition rectangular decomposition} on $(11,8,7,5,3,2,2,2)$ in the case of $a = 2$ and $b = 5$. In Figure \ref{fig37}, the rectangles in the $5$-rectangular decomposition are $r_1,r_2,$ and $r_3$. Since $\omega(\lambda) = 2$, $r_1^x = r_2^x = 2 = a$ and $r_3^x = 4 > a$.
\end{Example}

\begin{center}
\begin{figure}[h]
\begin{tikzpicture}[scale=0.5]

\draw (2,1) -- (2,9);
\draw (3,1) -- (3,9);
\draw (4,1) -- (4,9);
\draw (5,4) -- (5,9);
\draw (6,5) -- (6,9);
\draw (7,5) -- (7,9);
\draw (8,6) -- (8,9);
\draw (9,6) -- (9,9);
\draw (10,7) -- (10,9);
\draw (11,8) -- (11,9);
\draw (12,8) -- (12,9);
\draw (13,8) -- (13,9);
\draw (2,1) -- (4,1);
\draw (2,2) -- (4,2);
\draw (2,3) -- (4,3);
\draw (2,4) -- (5,4);
\draw (2,5) -- (7,5);
\draw (2,6) -- (9,6);
\draw (2,7) -- (10,7);
\draw (2,8) -- (13,8);
\draw (2,9) -- (13,9);
\draw [fill=yellow, opacity=0.2] (13,9) rectangle (9,7);
\draw [fill=yellow, opacity=0.2] (5,5) rectangle (9,7);
\draw [fill=yellow, opacity=0.2] (5,5) rectangle (3,1);
\node at (12.7,7.5) {$r_1$};
\node at (8.7,5.5) {$r_2$};
\node at (4.7,1.5) {$r_3$};

\end{tikzpicture}
\caption{Illustration of the statement of Proposition \ref{proposition rectangular decomposition} on $\lambda = (11,8,7,5,3,2,2,2)$, with $a=2$ and $b=5$. The yellow rectangles are $r_1,r_2,$ and $r_3$ of the $5$-rectangular decomposition.}
\label{fig37}
\end{figure}
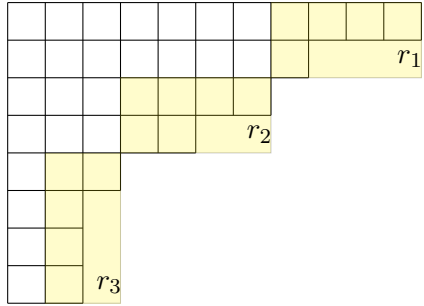
\end{center}

\begin{proof}[Proof of Proposition \ref{proposition rectangular decomposition}]
We first show that $r_1^x\geq a$. Suppose for the sake of contradiction that $r_1^x\leq a-1$ and $r_1^y\geq b-a+2$. Consider the ladder $\ell_{1,\lambda_1}$ and $\ell_{0,\lambda_1}$. Due to the dimensions of $r_1$, the box $(a+1,\lambda_1-(b-a))\in\ell_{1,\lambda_1}\setminus\lambda$ and $(a,\lambda_1-(b-a))\in\ell_{0,\lambda_1}\setminus\lambda$. However, this contradicts Lemma \ref{lemma: ladders gain where empty} since $(a+1,\lambda_1-(b-a))$ is the north-most box of $\ell_{1,\lambda_1}$ not contained in $\lambda$ and $(a,\lambda_1-(b-a)) \notin\lambda$ either. Thus, we must have $r_1^x\geq a$. We now consider the cases where $r_1^x> a$ and $r_1^x = a$ separately.

If $r_1^x> a$, Lemma \ref{lemma prohibiting shallow after steep rectangle} implies that for all $\alpha>1$ we have $r_\alpha^x>a$. In this case, $\omega(\lambda)=0$, and both parts of the proposition are satisfied.

If $r_1^x=a$, then $\omega(\lambda)\geq 1$. To prove the first part of the proposition, that for all $\alpha\in [1,\omega(\lambda)]$ the widths of the rectangles satisfy $r^x_\alpha = a$, we proceed by induction on $\alpha$.

Since $r_1^x=a$, the base case has already been done. Now suppose that $r^x_\alpha = a$ for $\alpha\in [1,\alpha']$ for a fixed $\alpha'<\omega(\lambda)$ and consider $r_{\alpha'+1}$. Denote the northeast-most box in $r_{\alpha'+1}$ by $(i,j)$. By the inductive hypothesis, all rectangles in the $b$-rectangular decomposition north of $r_{\alpha'+1}$ have width $r^x = a$, so the ladder $\ell_{i,j}$ must pass through the first row of $\lambda$, and $(i,j)$ is the north-most box in $\ell_{i,j}$ not contained in $\lambda$. Then proceeding as we did in the base case, if $r^x_{\alpha'+1}< a$, the two boxes $(i+a,j-(b-a))$ and $(i+a-1,j-(b-a))$ are both not in $\lambda$, which contradicts Lemma \ref{lemma: ladders gain where empty} (see Figure \ref{figpf35}). Thus, $r^x_{\alpha'} \geq a$. If $\alpha'<\omega(\lambda)$, then applying Lemma \ref{lemma prohibiting shallow after steep rectangle} on $r_{\alpha'}$ and $r_{\omega(\lambda)}$ forces $r^x_{\alpha'} \le r^x_{\omega(\lambda)} = a$. In particular, $r^x_{\alpha'} = a$ exactly, which completes the inductive step. 

\begin{center}
\begin{figure}[h]
\begin{tikzpicture}[scale=0.5]
\draw (0,0)--(0,2);
\draw (1,0)--(1,2);
\draw (2,0)--(2,2);
\draw (3,1)--(3,2);
\draw (4,1)--(4,2);
\draw (0,0) -- (2,0);
\draw (0,1) -- (4,1);
\draw (0,2) -- (4,2); 
\draw (-0.5,-2)--(5.5,2);
\foreach \x/\y/\m in {-0.5/1/\tiny{$a$},2/2.5/\tiny{$b-a+1$}} 
\node at (\x,\y) {\m};
\end{tikzpicture}
\quad\quad\quad
\begin{tikzpicture}[scale=0.5]
\draw (-1,1)--(-1,2);
\draw (0,1)--(0,2);
\draw (1,1)--(1,2);
\draw (2,1)--(2,2);
\draw (3,1)--(3,2);
\draw (4,1)--(4,2);
\draw (-1,1) -- (4,1);
\draw (-1,2) -- (4,2); 
\draw (-0.5,-2)--(5.5,2);
        \node at (-3.5,0.5) {\tiny{$(i+a-1,j-(b-a))$}};
\node at (-3,-0.5) {\tiny{$(i+a,j-(b-a))$}};
\node at (3.5,2.8) {\tiny{$(i,j)$}};
\draw[densely dotted] (0,0) -- (0,1) -- (1,1) -- (1,0) -- (0,0);
\draw[densely dotted] (0,-1) -- (0,0) -- (1,0) -- (1,-1) -- (0,-1);
\draw[->,black] (-0.5,0.5)--(0.5,0.5);
\draw[->,black] (-0.5,-0.5)--(0.5,-0.5);
\draw[->,black] (3.5,2.5)--(3.5,1.5);
\foreach \x/\y/\m in {-2/1.5/\tiny{$a-1$},1.0/2.5/\tiny{$b-a+2$}} 
\node at (\x,\y) {\m};
\end{tikzpicture}
\caption{Illustration of the inductive step for the case when $a = 2$ and $b = 5$. The left side shows a rectangle with width $r^x = a$, and the right side shows a rectangle with width $r^x < a$. In this case, the boxes $(i+a,j-(b-a))$ and $(i+a-1,j-(b-a))$ are both not in $\lambda$.}
\label{figpf35}
\end{figure}
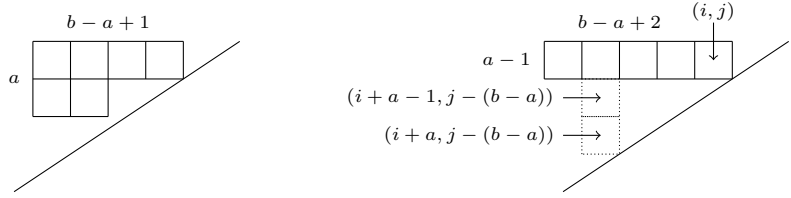
\end{center}

Finally, we show the second part of the proposition, that for all $\alpha\in[\omega(\lambda)+1,k]$ the widths of the rectangles satisfy $r^x_{\alpha} >a$. For any $\alpha$ in this interval, Lemma \ref{lemma prohibiting shallow after steep rectangle} implies that $r^x_{\alpha} \ge a$, and the definition of $\omega(\lambda)$ being the maximal index such that $r^x_{\alpha} = a$ forces $r^x_{\alpha} > a$.
\end{proof}

\begin{Lemma}\label{lemma later hooks not shallow}
Given a $\Cr_{a,b}$-valid partition $\lambda$ containing no hook of shape \eqref{leg=ta}, for $i>\psi(\lambda) $, any hook $H_{i,j}$ divisible by $b$ is not $(a,b)$-shallow.
\end{Lemma}

\begin{proof}
Let $H_{i,j}=tb$ for some $t$ and $i>\psi(\lambda)$. We assume for the sake of contradiction that $H_{i,j}$ is $(a,b)$-shallow. By Proposition \ref{proposition rectangular decomposition}, $r_{\omega(\lambda)+1}^x > a$. We show that among the $tb$ consecutive boxes on the rim that determine $H_{i,j}$, we can always find $b$ consecutive boxes which determine a rectangle $r_{\beta}$ such that $r_{\beta}^x \le a$. 

Number the $tb$ boxes in the rim corresponding to $H_{i,j}$ from northeast to southwest, and denote $s_\beta: \beta\in[1,t]$ to be the smallest rectangles each containing the $[(\beta-1)b+1 , \beta b]$ boxes in the set of $tb$ boxes. Since $H_{i,j}$ is $(a,b)$-shallow, $l_{i,j}\leq at-1$. The sum of the widths of all the rectangles cannot exceed $l_{i,j} + 1$, so we have $\sum_{\beta=1}^t (s_\beta^x-1)\leq at$. Hence, there exists a $\beta$, such that $s_{\beta}^x\leq a$. Applying Lemma \ref{lemma prohibiting shallow after steep rectangle} to $r_{\omega(\lambda)+1}$ and $r_{\beta}$ gives a hook of shape \eqref{leg=ta}, which is a contradiction. Therefore, there cannot be any $(a,b)$-shallow hooks strictly south of $x = \psi(\lambda).$
\end{proof}

\begin{Lemma}\label{lemma relating colseg with J}
Given a $\Cr_{a,b}$-valid and $b$-regular partition $\lambda$ containing no hook of shape \eqref{leg=ta}, we have the identity $$\lambda^{\Sr_{a,b}} = \lambda_{[1,\psi(\lambda)]}^{\J_b}\oplus\lambda_{[\psi(\lambda)+2,l(\lambda)]}.$$
\end{Lemma}

\begin{Example}
We illustrate Lemma \ref{lemma relating colseg with J} on $\lambda = (13,10,9,7,5,2,2,1)$ with $a=2$ and $b=5$. In the left half of Figure \ref{fig38}, we are applying $\Sr_{2,5}$. The rectangles of the $5$-rectangular decomposition are colored yellow, and the arrows indicate where boxes in the first row shift to. 

On the right half, we are applying $\J_5$. The $5$-rim is labeled, and the boxes that are shaded are removed. For the sake of illustration, we remove the south-most box in each column of $\lambda_{[\psi(\lambda)+1,l(\lambda)]}$, which is equivalent to removing $\lambda_{\psi(\lambda)+1}$. Visually, it is easy to see that both operations gives the partition $(10,10,7,6,2,2,1)$.
\end{Example}

\begin{center}
\begin{figure}[h]
\begin{tikzpicture}[scale=0.5]
\draw (15,0) -- (15,8);
\draw (16,0) -- (16,8);
\draw (17,1) -- (17,8);
\draw (18,3) -- (18,8);
\draw (19,3) -- (19,8);
\draw (20,3) -- (20,8);
\draw (21,4) -- (21,8);
\draw (22,4) -- (22,8);
\draw (23,5) -- (23,8);
\draw (24,5) -- (24,8);
\draw (25,6) -- (25,8);
\draw (26,7) -- (26,8);
\draw (27,7) -- (27,8);
\draw (28,7) -- (28,8);
\draw (15,0) -- (16,0);
\draw (15,1) -- (17,1);
\draw (15,2) -- (17,2);
\draw (15,3) -- (20,3);
\draw (15,4) -- (22,4);
\draw (15,5) -- (24,5);
\draw (15,6) -- (25,6);
\draw (15,7) -- (28,7);
\draw (15,8) -- (28,8);

\foreach \x/\y/\m in {+15.5/+0.5/$3$,+15.5/+1.5/$2$,+16.5/+1.5/$1$,+16.5/+2.5/$5$,+16.5/+3.5/$4$,+17.5/+3.5/$3$,18.5/+3.5/$2$,19.5/+3.5/$1$,+20.5/+4.5/$5$,+21.5/+4.5/$4$,+21.5/+5.5/$3$,22.5/+5.5/$2$,23.5/+5.5/$1$,+24.5/+6.5/$5$,+24.5/+7.5/$4$,+25.5/+7.5/$3$,26.5/+7.5/$2$,27.5/+7.5/$1$,} 
    \node at (\x,\y) {\m};
%more
\draw[pattern=north west lines, pattern color=black] (25,8) rectangle (28,7);
\draw[pattern=north west lines, pattern color=black] (22,6) rectangle (24,5);
\draw[pattern=north west lines, pattern color=black] (21,5) rectangle (22,4);
\draw[pattern=north west lines, pattern color=black] (17,4) rectangle (20,3);
\draw[pattern=north west lines, pattern color=black] (16,2) rectangle (17,1);
\draw[pattern=north west lines, pattern color=black] (15,1) rectangle (16,0);

\draw[densely dotted] (9,7) -- (10,7) -- (10,6) -- (9,6);
\draw[densely dotted] (5,5) -- (6,5) -- (6,4) -- (5,4);
\fill[black!30!white] (11.02,9.02) rectangle (12.98,7.98);
\draw[->] (13,8) -- (10,6);
\draw[->,black] (12,8) -- (6,4);
\draw[pattern=north west lines, pattern color=black] (0,9) rectangle (11,8);

\draw (0,1) -- (0,9);
\draw (1,1) -- (1,9);
\draw (2,2) -- (2,9);
\draw (3,4) -- (3,9);
\draw (4,4) -- (4,9);
\draw (5,4) -- (5,9);
\draw (6,5) -- (6,9);
\draw (7,5) -- (7,9);
\draw (8,6) -- (8,9);
\draw (9,6) -- (9,9);
\draw (10,7) -- (10,9);
\draw (11,8) -- (11,9);
\draw (12,8) -- (12,9);
\draw (13,8) -- (13,9);
\draw (0,1) -- (1,1);
\draw (0,2) -- (2,2);
\draw (0,3) -- (2,3);
\draw (0,4) -- (5,4);
\draw (0,5) -- (7,5);
\draw (0,6) -- (9,6);
\draw (0,7) -- (10,7);
\draw (0,8) -- (13,8);
\draw (0,9) -- (13,9);
\draw [fill=yellow, opacity=0.2] (13,9) rectangle (9,7);
\draw [fill=yellow, opacity=0.2] (5,5) rectangle (9,7);
\draw [fill=yellow, opacity=0.2] (5,5) rectangle (1,3);
\draw [fill=yellow, opacity=0.2] (2,3) rectangle (0,1);
\end{tikzpicture}
\caption{Illustration of Lemma \ref{lemma relating colseg with J} on $\lambda = (13,10,9,7,5,2,2,1)$ with $a=2$ and $b=5$. The shaded boxes are removed in the corresponding operator.}
\label{fig38}
\end{figure}
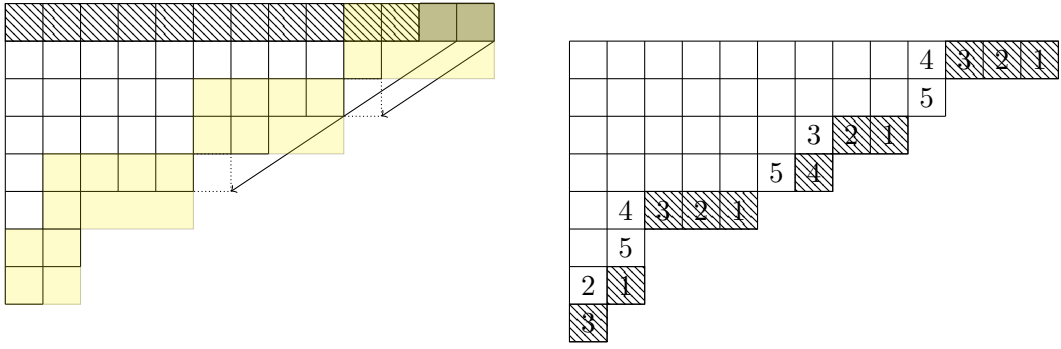
\end{center}

\begin{proof}[Proof of Lemma \ref{lemma relating colseg with J}]
We first claim that $\ell_{\psi(\lambda)+1,\lambda_{\psi(\lambda)+1}}\subset\lambda$, i.e. the ladder passing through the box $(\psi(\lambda)+1,\lambda_{\psi(\lambda)+1})$ is completely contained in $\lambda$. (if $\psi(\lambda)+1>l(\lambda)$, the argument below still holds by taking $\lambda_{\psi(\lambda)+1}=0$). 

By Proposition \ref{proposition rectangular decomposition}, all boxes on $\ell_{\psi(\lambda)+1,\lambda_{\psi(\lambda)+1}}$ that lie strictly north of $(\psi(\lambda)+1,\lambda_{\psi(\lambda)+1})$ are contained in $\lambda$. If $\lambda_{\psi(\lambda)+1}=0$, then the claim is immediate. Otherwise suppose on the contrary that there are boxes south of $(\psi(\lambda)+1,\lambda_{\psi(\lambda)+1})$ that are not contained in $\lambda$. Then consider the north-most box $(i,j)\in \ell_{\psi(\lambda)+1,\lambda_{\psi(\lambda)+1}}\setminus\lambda$. Since $r_1^x = \ldots = r_{\omega(\lambda)}^x = a,$ $\ell_{\psi(\lambda)+1,\lambda_{\psi(\lambda)+1}}$ passes through a box in the first row. By Lemma \ref{lemma: ladders gain where empty}, the box $(i-1,j)\in \lambda$. Then the hook $H_{\psi(\lambda)+1,j}=tb$, for some $t\in\mathbb{N}_{>0}$, must have $l_{\psi(\lambda)+1,j}=ta-1$, hence it is $(a,b)$-shallow, which contradicts Lemma \ref{lemma later hooks not shallow}. Thus, our initial assumption was false, and we conclude that $\ell_{\psi(\lambda)+1,\lambda_{\psi(\lambda)+1}}\subset\lambda$.

In the case where $\omega(\lambda)=\psi(\lambda)=0$, the ladder $\ell_{1,\lambda_1}$ is completely full so no boxes move when applying $\Sr_{a,b}$. Hence $\lambda^{\Sr_{a,b}}=\lambda_{[2,l(\lambda)]}$, which is the desired identity.

In the case where $\omega(\lambda)\geq1$, and $\psi(\lambda)\geq a$, we know from above that no changes occur south of row $\psi(\lambda)+1$ since the ladder $\ell_{\psi(\lambda)+1,\lambda_{\psi(\lambda)+1}}$ passes through the first row and has all its boxes contained in $\lambda$. So we must have $$\lambda^{\Sr_{a,b}}_{[\psi(\lambda)+1,l(\lambda^{\Sr_{a,b}})]} = \lambda_{[\psi(\lambda)+2,l(\lambda)]}. $$

Consider the description of the $b$-rectangular decomposition stated in Proposition \ref{proposition rectangular decomposition}. For $\alpha=1,\ldots,\omega(\lambda)$, the northeast-most box in $r_\alpha$ (which is contained in $\lambda$), has coordinates $(1+(\alpha-1)a,\lambda_{1+(\alpha-1)a})$. From the shape of the $r_\alpha$'s given by Proposition \ref{proposition rectangular decomposition}, we know that $$\lambda_{1+\alpha a}\leq \lambda_{1+(\alpha-1)a}-(b-a),\quad \alpha=1,\ldots\omega(\lambda).$$
Denote $$\delta_{\alpha}:=\lambda_{1+(\alpha-1)a}-(b-a)-\lambda_{1+\alpha a},\quad \alpha=1,\ldots\omega(\lambda).$$
We define the $b$-gaps of $\lambda$ as a collection of boxes in the first $\psi(\lambda)+1$ rows of $\lambda$ and immediately southeast to a box in $\{\text{rim of }\lambda\}\setminus \{b\text{-rim of }\lambda\}$ (see Figure \ref{fig310}). More specifically, we let
$$G_{\alpha}(\lambda)=\{(a\alpha+1,j): a\in [1,\omega(\lambda)],(a\alpha-1,j-1)\in \text{rim of }\lambda\setminus b\text{-rim of }\lambda\}.$$

\begin{center}
\begin{figure}
\begin{tikzpicture}[scale=0.5]
\draw (0,0) -- (0,4);
\draw (1,0) -- (1,4);
\draw (2,1) -- (2,4);
\draw (3,1) -- (3,4);
\draw (4,2) -- (4,4);
\draw (5,2) -- (5,4);
\draw (6,3) -- (6,4);
\draw (7,3) -- (7,4);
\draw (8,3) -- (8,4);
\draw (0,0) -- (1,0);
\draw (0,1) -- (3,1);
\draw (0,2) -- (5,2);
\draw (0,3) -- (8,3);
\draw (0,4) -- (8,4);
\draw [densely dotted] (4,1)--(5,1)--(5,2)--(4,2)--(4,1);
\draw [densely dotted] (3,1)--(4,1)--(4,2)--(3,2)--(3,1);
\draw (8,3) -- (2,-1) ;
\draw (7,3) -- (1,-1) ;
\draw [fill=yellow, opacity=0.2] (4,3) rectangle (2,2);
\node at (4,-0.5) {$\ell_{1,8}$};
\node at (0.5,-0.5) {$\ell_{1,7}$};
\draw[->,black] (2.5,2.5) -- (3.5,1.5);
\draw[->,black] (3.5,2.5) -- (4.5,1.5);
\end{tikzpicture}
\caption{Illustration of $G_1$ for $(8,5,3,1)$ when $a = 2$ and $b = 5$. The boxes in $\{\text{rim of }\lambda\}\setminus \{5\text{-rim of }\lambda\}$ are colored in yellow and each point to a corresponding box of $G_1$. Notice that the boxes in the first row of $\ell_{1,7}$ and $\ell_{1,8}$ slide exactly to the boxes of $G_1$ when applying $\Sr_{2,5}$.}
\label{fig310}
\end{figure}
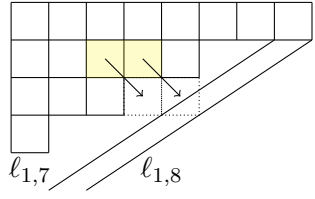
\end{center}

By Proposition \ref{proposition rectangular decomposition}, ladders passing through boxes in the first row also pass through some box in the first row of $r_\alpha: \alpha\in[1,\omega(\lambda)]$ and the row $\psi(\lambda)+1$. Hence, for $\alpha\in[1,\omega(\lambda)]$, $G_{\alpha}(\lambda)$ are exactly the $\delta_{\alpha}$ boxes $(1+\alpha a, \lambda_{1+\alpha a}+k_\alpha): k_\alpha\in[1,\delta_\alpha]$ that become occupied under $\Sr_{a,b}$. Applying $\Sr_{a,b}$ to $\lambda$ precisely slides boxes $(1,\lambda_1),\ldots, (1,\lambda_1-\sum_{\beta=1}^{\omega(\lambda)}\delta_\beta)$ in order to the boxes in $G_1(\lambda)\sqcup\ldots\sqcup G_{\alpha}$ when sorted in order from northeast to southwest, and $\Sr_{a,b}$ finishes by removing what remains in the first row.

Using the definition of $\J_b$, we know that $\lambda^{\J_b}_i=\lambda_{i+1}$ for all $i\in\{1,\ldots,\psi(\lambda)\}: i\neq \beta a$ for some $\beta\in [1,\omega(\lambda)]$. And for all $\beta\in [1,\omega(\lambda)]$, we have $\lambda^{\J_b}_{\beta a}=\lambda_{\beta a+1}+\delta_{\beta}$. Thus, we conclude that $\lambda^{\Sr_{a,b}}_{[1,\psi(\lambda)]} = \lambda^{\J_b}_{[1,\psi(\lambda)]}.$

Since they agree in every part, $\lambda^{\Sr_{a,b}} = \lambda_{[1,\psi(\lambda)]}^{\J_b}\oplus\lambda_{[\psi(\lambda)+2,l(\lambda)]},$ as desired.
\end{proof}

\begin{Remark}
Before proving Theorem \ref{theorem generalizing BOX}, we again emphasize that in all previous results, we have only used the conditions that $\lambda$ is $\Cr_{a,b}$-valid and $\lambda$ contains no hooks of shape \eqref{leg=ta}. In the next section, we will claim that these results can be used towards proving a generalization of Theorem \ref{theorem generalizing BOX}. 
\end{Remark}

To finally prove Theorem \ref{theorem generalizing BOX}, we now for the first time use the full extent of the condition that $\lambda$ contains no $(a,b)$-shallow hooks. The theorem is restated below for reference:

\mainthm*

\begin{proof}[Proof of Theorem \ref{theorem generalizing BOX}]
First, we note that $\lambda$ is $b$-regular. If on the contrary, there exists $i$ such that $\lambda_i = \ldots = \lambda_{i+b-1}>\lambda_b$, then the hook $H_{i,\lambda_i}$ is divisible by $b$ and has leg $b-1$ and arm $0$, which is not $(a,b)$-shallow. Thus, $\lambda$ must be $b$-regular.

We proceed by induction on $|\lambda|.$ In the base case, we have $(1)^{\M_b\Tr} = (1)^{\Cr_{a,b}} = (1)$ for all $a,b$. 

Consider a partition $\lambda$ that is $\Cr_{a,b}$-valid with all hooks divisible by $b$ being $(a,b)$-shallow. We assume that $\mu^{\M_b\Tr}=\mu^{\Cr_{a,b}}$ holds for all $\mu$ with the same constraints and $|\mu|<|\lambda|$.

We claim that the $b$-rim strictly south of row $\psi(\lambda)$ forms a single $b'$-segment. If there is a $b$-segment strictly south of $x = \psi(\lambda)$, it must be $(a,b)$-shallow by the conditions of the theorem. However, Lemma \ref{lemma later hooks not shallow} implies that such an $(a,b)$-shallow hook cannot exist. Thus, the $b$-rim of $\lambda$ strictly south of $x = \psi(\lambda)$ is a single $b'$-segment. Hence $\lambda^{\J_b}_{[\psi(\lambda)+1,l(\lambda^{\J_b})]} = \lambda_{[\psi(\lambda)+2,l(\lambda)]}$. Lemma \ref{lemma relating colseg with J} then implies $\lambda^{\J_b} = \lambda^{\Sr_{a,b}}$, so $$(\lambda^{\Cr})_1 = |\lambda|-|\lambda^{\Sr_{a,b}}| = |\lambda|-|\lambda^{\J_b}| = (\lambda^{\M_b\Tr})_1=|\{y: 1\leq y\leq \lambda_1,\ell_{1,y}\subset\lambda\}|.$$

In order to apply the inductive hypothesis, it remains to show that $\lambda^{\J_b} = \lambda^{\Sr_{a,b}}$ is $\Cr_{a,b}$-valid and has all hooks divisible by $b$ being $(a,b)$-shallow. The $\Cr_{a,b}$-validity is clear since $\Cr_{a,b}$ is defined recursively by $\Sr_{a,b}$ via equation  \eqref{eq: Cr recursion} and $\lambda$ is $\Cr_{a,b}$-valid to begin with. 

It remains to show that all hooks divisible by $b$ in $\lambda^{\Sr_{a,b}}$ are $(a,b)$-shallow. Consider all hooks $H_{i,j}({\lambda^{\Sr_{a,b}}})=tb$ for some $t\in\mathbb{N}_{>0}$ and compare them with $H_{i+1,j}({\lambda})$ (recall that when performing $\Sr_{a,b}$, the first row is removed, so there is index shifting). If these two hooks have the same dimension, then $H_{i,j}({\lambda^{\Sr_{a,b}}})$ is $(a,b)$-shallow because all hooks divisible by $b$ in $\lambda$ are $(a,b)$-shallow. They can only be different in the following three situations:
\begin{enumerate}
\item The foot box $\mathfrak{f}_{i,j}(\lambda^{\Sr_{a,b}})=(i',j)$ comes from a box in the first row of $\lambda$, but the hand box does not.

Suppose on the contrary that $H_{i,j}({\lambda^{\Sr_{a,b}}})$ is not $(a,b)$-shallow, the corresponding leg $l_{i,j}({\lambda^{\Sr_{a,b}}})\geq at$. Then $(i'-ta,j+t(b-a))$, which is on the ladder $\ell_{i',j}$, is not contained in $\lambda^{\Sr_{a,b}}$ and north of $(i',j)$, contradicting the definition of $\Sr_{a,b}$. Thus, $H_{i,j}({\lambda^{\Sr_{a,b}}})$ must be $(a,b)$-shallow.

\item The hand box $\mathfrak{h}_{i,j}(\lambda^{\Sr_{a,b}})=(i,j')$ comes from a box in the first row of $\lambda$, but the foot box does not.

Since boxes only slide to rows that are multiples of $a$ under $\Sr_{a,b}$, in this case $i=\alpha a-1$ for $\alpha\geq 1$. Suppose on the contrary $H_{i,j}({\lambda^{\Sr_{a,b}}})$ is not $(a,b)$-shallow, then we claim that $H_{i-a+1,j}(\lambda)$ is also divisible by $b$ and not $(a,b)$-shallow either. 

$H_{i-a+1,j}(\lambda)$ is also divisible by $b$ is clear because the rectangle in the $b$-rectangular decomposition of $\lambda$ occupying rows $i-a+1$ to $i$ has $r^x = a$ and $r^y=b-a+1$ as shown in Figure \ref{figproof12sufficiency}. Then $H_{i-a+1,j}(\lambda)=(t+1)b$.

\begin{center}
\begin{figure}[h]
\begin{tikzpicture}[scale=0.5]
\draw[brown] (0,0) rectangle (5,3);
\draw[brown] (7,3) rectangle (12,6);
\draw (3,3) rectangle (8,2);
\draw (3,-1) rectangle (4,6);
\draw (3,5) rectangle (12,6);
\draw[line width=1pt] (3,-1)--(4,-1)--(4,2)--(8,2)--(8,3)--(3,3)--(3,-1);
%\draw[pattern=north west lines, pattern color=black] (3,-1) rectangle (4,3);
\draw[pattern=north west lines, pattern color=black] (5,2) rectangle (8,3);
\draw [fill=red, opacity=0.2] (3,-1) rectangle (4,6);
\draw [fill=red, opacity=0.2] (4,5) rectangle (12,6);
\draw[->] (2.7,2.5)--(3.3,2.5);
%\draw[->] (2.7,5.5)--(3.3,5.5);
\node at (2,2.4) {$(i,j)$};
\end{tikzpicture}
\caption{The hand box $\mathfrak{h}_{i,j}(\lambda^{\Sr_{a,b}})=(i,j')$ comes from a box in the first row of $\lambda$, but the foot box does not. The boxes of the $b$-gap are shaded in the figure, which comes from the first row of $\lambda$. $H_{i,j}(\lambda^{\Sr_{a,b}})$ is outlined by black lines and $H_{i-a+1,j}(\lambda)$ is colored in red.}
\label{figproof12sufficiency}
\end{figure}
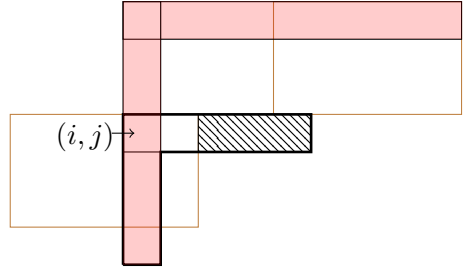
\end{center}

Then, $l_{i-a+1,j}(\lambda)=l_{i,j}(\lambda^{\Sr_{a,b}})+a\geq (t+1)a$, so $H_{i-a+1,j}(\lambda)=(t+1)b$ is not $(a,b)$-shallow, a contradiction. Thus, we conclude $H_{i,j}(\lambda^{\Sr_{a,b}})$ must be $(a,b)$-shallow.

\item Both the foot box $\mathfrak{f}_{i,j}(\lambda^{\Sr_{a,b}})=(i',j)$ and the hand box $\mathfrak{h}_{i,j}(\lambda^{\Sr_{a,b}})=(i,j')$ come from boxes in the first row of $\lambda$.

Proposition \ref{proposition rectangular decomposition} implies that in this case both $i$ and $i'$ are multiples of $a$, so we let $l_{i,j}(\lambda^{\Sr_{a,b}})=i'-i= a\alpha$. But we also have $a_{i,j}(\lambda^{\Sr_{a,b}})=j'-j\geq (b-a)\alpha+1$ since there are $\alpha$ rectangles between $y = j$ and $y=j'$. In particular, these inequalities imply that $H_{i,j}(\lambda^{\Sr_{a,b}})$ must be $(a,b)$-shallow.

\end{enumerate}
Now we can apply the inductive hypothesis: $\lambda^{\J_b\M_b\Tr} = \lambda^{\Sr_{a,b}\Cr_{a,b}}.$ By Equation \eqref{eq: Cr recursion} and Equation \eqref{eq: M recursion}, $\lambda^{\Cr_{a,b}} = (|\lambda|-|\lambda^{\Sr_{a,b}}|)\oplus \lambda^{\Sr_{a,b}\Cr_{a,b}}$ and $\lambda^{\M_b\Tr} = (|\lambda|-|\lambda^{\J_b}|)\oplus \lambda^{\J_b\M_b\Tr},$ so we find that $\lambda^{\M_b\Tr} = \lambda^{\Cr_{a,b}}$. 
\end{proof}

\begin{Remark}
Note that the proofs in this section do not require $a$ and $b$ to be co-prime.
\end{Remark}

\section{Iwahori-Hecke Algebras and Conjectures}
In this section we state the following two conjectures, which have been proven in the case of $a=1$ by Bessenrodt, Olsson, and Xu in \cite{bessenrodt1999properties} and Fayers in \cite{fayers2008regularisation} respectively:
\begin{Conjecture}\label{reverse1.2}
Given positive co-prime integers $a<b$ and a $b$-regular, $\Cr_{a,b}$-valid partition $\lambda$, if $\lambda^{\M_b\Tr} = \lambda^{\Cr_{a,b}}$, then all hooks $H_{i,j}$ in $\lambda$ with $b\mid H_{i,j}$ must satisfy Equation \eqref{equation: shallow}:
\begin{equation*}
    \left(\frac{b}{a}-1\right)l_{i,j} < a_{i,j} +1.
\end{equation*}
\end{Conjecture}

\begin{Conjecture}
\label{theorem generalizing Fayers}
Given positive integers $a$, $b$ with $2a<b$ and a partition $\lambda$ that is both $\Reg_{a,b}$-valid and $\Cr_{a,b}$-valid, we have the following:
\begin{enumerate}
\item If all hooks $H_{i,j}$ in $\lambda$ with $b\mid H_{i,j}$ satisfy:
\begin{equation}\label{equation: steep and shallow}
   \text{either} \quad\left(\frac{b}{a}-1\right)l_{i,j} < a_{i,j} +1\quad \text{or}\quad \left(\frac{b}{a}-1\right)a_{i,j} < l_{i,j} +1,
\end{equation}
then $\lambda^{\Reg_{a,b}\M_b} = \lambda^{\Tr\Reg_{a,b}}$ (i.e. $\lambda^{\Reg_{a,b}\M_b\Tr} = \lambda^{\Cr_{a,b}})$.
\item If $a$ and $b$ are co-prime, and $\lambda$ satisfies $\lambda^{\Reg_{a,b}\M_b} = \lambda^{\Tr\Reg_{a,b}}$, then all hooks divisible by $b$ satisfy Equation \eqref{equation: steep and shallow}.
\end{enumerate}
\end{Conjecture}

%%%%%
Conjecture \ref{reverse1.2} and part (2) of Conjecture \ref{theorem generalizing Fayers} are closely related to the $q$-decomposition numbers  and we first recall the main facts about Iwahori-Hecke algebras that motivates the representation-theoretic relationship of Mullineux involution and column regularization.
%which is the $\mathbb{F}$-algebra with generators $T_1,\ldots, T_{n-1}$ satisfying:
%\begin{align*}
%(T_i-q)(T_i+1) & =0 & \text{for } i =1,\ldots n-1\\
%T_iT_j & =T_jT_i  & \text{for } 1\leq i<j-1\leq n-2\\
%T_iT_{i+1}T_i & =  T_{i+1}T_iT_{i+1}   & \text{for } i =1,\ldots n-2
%\end{align*}

Let $\mathbb{F}$ be a field and $q\in\mathbb{F}\setminus \{0\}$ and denote $\mathcal{H} = \mathcal{H}_{\mathbb{F},q}(S_n)$ to be the Iwahori-Hecke algebra of $S_n$. We refer to \cite{mathas1999iwahori} for definitions and representation theory of $\mathcal{H}$. Let $b$ be the minimal integer satisfying $1+q+\cdots+q^{b-1}=0$ with $b\ge 2.$ For every partition $\lambda$ of $n$, there is a Specht module $S^{\lambda}$ of $\mathcal{H}$, and if $\lambda$ is $b$-regular, there is a simple module $D^\lambda$ which is the co-socle of $S^{\lambda}$. The \textit{decomposition numbers} $d_{\lambda,\mu}=\left[S^\lambda:D^\mu\right]$ satisfy the following important properties:

\begin{Proposition}[\cite{james1976decomposition,Lascoux1996}]
The decomposition numbers $d_{\lambda,\mu}$ satisfies the following:
\begin{enumerate}
\item $d_{\lambda,\mu} = 0 $ unless $\lambda$ and $\mu$ have the same $b$-core;
\item $d_{\lambda,\mu} = 0 $ unless $\mu\unlhd\lambda^{\Reg_{1,b}}$ and $d_{\lambda,\lambda^{\Reg_{1,b}}}=1$;
\item If $\mu$ is $b$-regular, $d_{\lambda,\mu}=d_{\lambda^{\Tr},\mu^{\M_b}}$.
\end{enumerate}

\end{Proposition}

The $q$-analogue of decomposition numbers $d_{\lambda,\mu}(q)$ is defined by considering the graded multiplicity of $D^\mu$ inside $S^\lambda$ and we refer to \cite{kleshchev2010representation} for details regarding its definitions and properties. Lascoux, Leclerc and Thibon conjectured in \cite{Lascoux1996} that when $\mathbb{F}=\mathbb{C}$, the $d_{\lambda,\mu}(q)$'s are the same coefficients in the expansion of the lower global crystal basis $\{G(\mu)\mid\mu\in\mathcal{P} \text{ is }b\con\text{regular}\}$ explained as follows. This result was proved by Ariki in \cite{ariki1996decomposition}.

The readers may refer to \cite{Lascoux1996} for details on the quantized affine Lie algebra $U_q(\hsl_b)$ and its action on the \textit{Fock space} $\mathcal{F}=\bigoplus_{\lambda\in\mathcal{P}}\mathbb{Q}(q)\lambda$, which was originally developed by Misra-Miwa \cite{misra1990crystal} using the work of Hayashi \cite{hayashi1990q}. In our discussion, we state without proof Kashiwara's existence and uniqueness of crystal bases of the integrable highest weight modules of affine quantum algebra:

\begin{Theorem}[\cite{kashiwara1991crystal}]
\label{kashiwara}
Denote $M(\Lambda_0)_{\mathbb{Q}}=U_{\mathbb{Q}}^-\emptyset$ where $U^-_{\mathbb{Q}}\subset U_q(\hsl_b)$ is the $\mathbb{Q}[q,q^{-1}]$-sub-algebra generated by the Chevalley generators $f_i^{(m)}$ of $U_q(\hsl_b)$. Then there exists a unique $\mathbb{Q}[q,q^{-1}]$-basis (which is called the lower crystal basis) of $M(\Lambda_0)_{\mathbb{Q}}$:
$$\{G(\mu):\mu \text{ is } b\text{-regular}\},$$ such that the following holds:
\begin{enumerate}
\item $G(\mu)\equiv\mu\text{ mod }qL$;
\item $\overline{G(\mu)}=G(\mu)$.
\end{enumerate}
\end{Theorem}

We expand $G(\mu)$ using the standard basis $\mathcal{P}$:
$$G(\mu)=\sum_{\lambda}d_{\lambda,\mu}(q)\lambda.$$
The \textit{$q$-decomposition numbers} $d_{\lambda,\mu}(q)$ satisfies the following properties:

\begin{Theorem}[\cite{Lascoux1996,varagnolo1999decomposition,schiffmann2000hall}]
\label{LLTcoefficients}
The $q$-decomposition numbers satisfy the following:
\begin{enumerate}
\item $d_{\lambda,\mu}(q)\in\mathbb{N}[q]$;
\item $d_{\mu,\mu}(q)=1$ and $d_{\lambda,\mu}(q)=0$ unless $\lambda\unlhd\mu$, $|\lambda|=|\mu|$, and $\Core_b(\lambda)=\Core_b(\mu)$;
\item $d_{\lambda^{\Tr},\mu^{\M_b}}(q) = q^{|\mu|_b}d_{\lambda,\mu}(q^{-1})$, where $|\mu|_b$ is the $b$-weight of $\mu$, which is the number of hooks in $\mu$ that is divisible by $b$. 
\end{enumerate}
\end{Theorem}

The third property in the above theorem is important in relating the Mullineux involution with the $q$-decomposition numbers.

\begin{Conjecture}
\label{Fayers2.2}
Suppose that $\lambda$ and $\mu$ are partitions of the same size and $\lambda$ is $\Reg_{a,b}$-valid and $\mu$ is $(a,b)$-regular. Then $d_{\lambda,\lambda^{\Reg_{a,b}}}(q) = q^{s(\lambda)}$, where $s(\lambda)$ is the number of $(a,b)$-steep hooks in $\lambda$ that are divisible by $b$.
\end{Conjecture}

Conjecture \ref{Fayers2.2} generalizes a part of Fayers's main theorem \cite[Theorem 2.2]{fayers2007q} in terms of the generalized (resp. column) regularization. In fact, this is the only ingredient needed for Conjecture \ref{reverse1.2} and part (2) of Conjecture \ref{theorem generalizing Fayers}. Assume that Conjecture \ref{Fayers2.2} holds, then we have:

\begin{proof}[Proof of Conjecture \ref{reverse1.2}]
Suppose a $\Cr_{a,b}$-valid partition $\lambda$ satisfies $\lambda^{\M_b\Tr} = \lambda^{\Cr_{a,b}}$, then $$d_{\lambda^{\Tr},\lambda^{\M_b}}(q) = d_{\lambda^{\Tr},\lambda^{\Tr\Reg_{a,b}}}(q).$$ By Conjecture \ref{Fayers2.2}, $$d_{\lambda^{\Tr},\lambda^{\Tr\Reg_{a,b}}}(q) = q^{s(\lambda^{\Tr})},$$ where $s(\lambda^{\Tr})$ is also the number of $(a,b)$-shallow hooks in $\lambda.$ 

From Theorem \ref{LLTcoefficients}, we know that $$d_{\lambda^{\Tr},\lambda^{\M_b}} = q^{|\lambda|_b}d_{\lambda,\lambda}(q^{-1}) = q^{|\lambda|_b}.$$ Therefore, $|\lambda|_b= s(\lambda^{\Tr}),$ so all hooks in $\lambda$ divisible by $b$ must be $(a,b)$-shallow.  
\end{proof}

\begin{proof}[Proof of part (2) of Conjecture \ref{theorem generalizing Fayers}]
Suppose we have a both $\Cr_{a,b}$-valid and $\Reg_{a,b}$-valid partition $\lambda$ satisfying $\lambda^{\Reg_{a,b}\M_b} = \lambda^{\Tr\Reg_{a,b}}$. Then we have $$d_{\lambda^{\Tr},\lambda^{\Tr\Reg_{a,b}}} (q)= d_{\lambda^{\Tr},\lambda^{\Reg_{a,b}\M_b}}(q).$$ 
On the one hand, by Conjecture \ref{Fayers2.2}, $$d_{\lambda^{\Tr},\lambda^{\Tr\Reg_{a,b}}} (q)= q^{s(\lambda^{\Tr})}.$$ 

On the other hand, $$d_{\lambda^{\Tr},\lambda^{\Reg_{a,b}\M_b}}(q)=q^{|\lambda^{\Reg_{a,b}}|_b}d_{\lambda,\lambda^{\Reg_{a,b}}}(q^{-1})=q^{|\lambda|_b}d_{\lambda,\lambda^{\Reg_{a,b}}}(q^{-1})=q^{|\lambda|_b}q^{-s(\lambda)},$$
where the first equality follows from Theorem \ref{LLTcoefficients}, the second from Lemma \ref{reg preserves core} and the last one from Conjecture \ref{Fayers2.2}. Therefore, we find that $s(\lambda) + s(\lambda^{\Tr}) = |\lambda|_b$. Lemma \ref{simultaneously shallow/steep} indicates no hooks can be both $(a,b)$-shallow and $(a,b)$-steep in case of $2a<b$, which means all hooks divisible by $b$ must be either shallow or steep.
\end{proof}

%Finally we would like to point out that the $(a, b)$-regular partitions form an important collection, whose algebraic interpretations will be studied in the later papers of this series.

%\newpage
\bibliographystyle{alpha}
\bibliography{MTCOLREG}

\end{document}